\documentclass[12pt]{article}
\usepackage{amsmath,amssymb,comment}
\usepackage{amsfonts}
\usepackage{latexsym}
\usepackage{epsfig}  
\usepackage{graphicx}
\usepackage{MnSymbol}
\usepackage{url}
\usepackage{color}
\usepackage[caption=false]{subfig}

\usepackage{cleveref}
\newtheorem{theorem}{Theorem}[section]

\newtheorem{lemma}[theorem]{Lemma}

\newtheorem{proposition}[theorem]{Proposition}
\newtheorem{remark}[theorem]{Remark}

\newenvironment{proof}[1][Proof]{\textbf{#1.} }{\ \rule{0.5em}{0.5em}}
\def\R{\mathbb{R}}

\def\T{\mathbb{T}}

\def\d{\varepsilon}

\def\s{\sigma}

\def\dd{a}
\renewcommand{\O}{\mathcal{O}}
\newcommand{\uns}{\mathrm{u}}
\newcommand{\sta}{\mathrm{s}}
\newcommand{\hz}{H\!Z}
\newcommand{\cc}{\Lambda}
\newcommand{\J}{\mathcal{J}}

\title{
Hopf-Zero singularities truly unfold chaos
}

\begin{document}

\title{
Hopf-Zero singularities truly unfold chaos
}

\author{I. Baldom\'a \footnote{Universitat Polit\`ecnica de Catalunya,
              Barcelona Graduate School of Mathematics (BGSMATH),
              immaculada.baldoma@upc.edu  }, S. Ib\'a\~{n}ez \footnote{Universidad de Oviedo,
              mesa@uniovi.es}, T.M. Seara \footnote{Universitat Polit\`ecnica de Catalunya,
            Barcelona Graduate School of Mathematics (BGSMATH),
            tere.m-seara@upc.edu}.}


\date{\today}
\maketitle

\begin{abstract}
We provide conditions to guarantee the occurrence of Shilnikov bifurcations in analytic unfoldings of some
Hopf-Zero singularities through a beyond all order phenomenon: the exponentially small breakdown of invariant manifolds which coincide at any order of the normal form procedure. The conditions are computable and satisfied by generic singularities and generic unfoldings.

The existence of Shilnikov bifurcations in the $\mathcal{C}^r$ case was already argued by Guckenheimer in the $80$'s. About the same time, endowing the space of $\mathcal{C}^\infty$ unfoldings with a convenient topology, persistence and density of the Shilnikov phenomenon was proved by Broer and Vegter in 1984. However, since the proof involves the use of flat perturbations, this approach is not valid in the analytic context. What is more, none of the mentioned approaches provide a computable criteria to decide whether a given unfolding exhibits Shilnikov bifurcations or not.

This work ends the whole discussion by showing that, under generic and checkable hypothesis, any analytic unfolding of a Hopf-Zero singularity within the appropriate class contains Shilnikov homoclinic orbits, and as a consequence chaotic dynamics.
\end{abstract}

\section{Introduction}

Results guaranteeing the existence of chaotic dynamics are hardly available in the literature.
Although there is a plenty of examples based on numerical evidences, analytical proofs are rare.
It is only recently that simple criteria are available for certain sets of dynamical systems as, for instance, unfoldings of singularities of vector fields; the general framework of this paper.

The route through dynamical complexity starts with Poincar\'e. In his seminal essay \cite{Poin1890},
he discovered that homoclinic intersections between the invariant manifolds of a hyperbolic fixed point were sources of very complicated behaviours.
Later  \cite{Bir35}, Birkhoff showed that, for  planar diffeomorphisms, near a transverse homoclinic intersection there exists an extremely intricate set of periodic orbits.
By the mid 60's, Smale \cite{Sma67} conceived his celebrated horseshoe to explain, via conjugations to Bernoulli shifts,
the Birkhoff result and also additional features of the complicated dynamics arising near a homoclinic intersection.
In \cite{MV93}, Mora and Viana proved the appearance of strange attractors when tangent homoclinic intersections were unfolded in families of planar diffeomorphisms.
These attractors are like those in \cite{BC91} for the H\'{e}non family, that is, they are non hyperbolic and persistent in the sense of measure.

On the other hand, Shilnikov \cite{Shil65} proved the counterpart of the Birkhoff results in the context of smooth vector fields on $\mathbb{R}^{3}$.
Namely, he proved the existence of a countable set of periodic orbits in every neighbourhood of a homoclinic orbit to a hyperbolic equilibrium point with eigenvalues
$\lambda $ and $-\varrho \pm\omega i$, with $0<\varrho<\lambda $.
Because of the resemblance with the analogous statement for transverse homoclinic points in diffeomorphisms, one would expect to find
the Smale horseshoes playing a key role.
And so it was that Tresser \cite{Tresser} showed that in every neighbourhood of a Shilnikov homoclinic orbit,
an infinite number of linked horseshoes can be defined in such a way that the dynamics is conjugated to a subshift of finite type on an infinite number of symbols.
Moreover, when the homoclinic connection is generically unfolded, disappearance of horseshoes is accompanied by unfoldings of homoclinic tangencies to periodic orbits,
leading to persistent non hyperbolic strange attractors like those in \cite{MV93}.
In \cite{Puma97,Puma01,homburg2002}, it was proven that infinitely many of these attractors can coexist for non generic families of vector fields.
For an extensive study of the phenomena accompanying homoclinic bifurcations see \cite{PalisTakens,bondiavia}.
Shilnikov homoclinic orbits are the simplest global configurations which, as just argued, can explain the existence of chaotic dynamics in families of vector fields,
but, unfortunately, their existence is again not easy to prove in a given system.

Finally, in this searching of friendly criteria, we come into the world of singularities.
They are, needless to say, manageable objects, much more manageable than global structures such as homoclinic intersections.
It has been proven that Shilnikov saddle-focus homoclinic orbits appear in generic smooth unfoldings of certain singularities.
Indeed, in \cite{ibarod1995} it was proven that these configurations were unfolded in generic unfoldings of three-dimensional nilpotent singularities of codimension four.
Degeneracy was reduced in \cite{ibarod2005}, where Shilnikov homoclinic bifurcations were proven to exist in any generic unfolding of a three-dimensional
nilpotent singularity of codimension three.
Therefore, suspended H\'enon like strange attractors, appear in generic unfoldings of such singularities.
See \cite{BIR11} for additional technical details and \cite{dumibakok2001,dumibakok2006} for complementary results regarding the rich unfolding
of the three-dimensional nilpotent singularity. In higher dimensions there also exists singularities of low codimension which can play the role of organizing centers of chaotic dynamics. For instance, \cite{DIR2019} provides numerical evidences of the existence of strange attractors in the unfolding of the codimension-three Hopf-Bogdanov-Takens bifurcation. This bifurcation has also been studied in the context of Hamiltonian systems (see \cite{JBL2016}).

One of the main motivations for obtaining simple criteria for the existence of strange attractors in given families of vector fields is their applicability.
Results in \cite{ibarod2005} have been successfully used to prove the existence of chaotic dynamics in a model of coupled oscillators (see \cite{DIR2007}),
in a general class of delay differential equations (see \cite{CY2008}) and also in a three-dimensional predator-prey model (see \cite{DIP2019}).

Once it has been proven that chaos is unfolded by low codimension singularities, a natural question arises: which is the lowest codimension
level where there exist singularities unfolding strange attractors or, in other words, what is the simplest local bifurcation displaying strange attractors.
Obviously, one is not such lowest codimension because the only codimension one bifurcations are the saddle-node bifurcation of equilibrium points and the Hopf bifurcation,
none of them including chaos.
So, if strange attractors are unfolded from codimension three singularities, the question is: do there exist codimension-two local bifurcations unfolding generically strange attractors?

In this paper we consider Hopf-Zero ($\hz$ in the sequel) singularities, that is,  three-dimensional vector fields $X^*$ such that
$X^*(0,0,0)=0$ and $DX^*(0,0,0)$ has eigenvalues $\pm i \alpha^*$ and $0$, with $\alpha* > 0$.
In fact, without loss of generality, we can assume that:
\begin{equation}
\label{hopf_zero_linear}
DX^*(0,0,0)=\left(\begin{matrix}
0 & \alpha^* & 0\\
-\alpha^* & 0 & 0\\
0 & 0 & 0
\end{matrix}
\right).
\end{equation}

Classification of $\hz$ singularities was first done by Takens \cite{Tak74}.
Up to an analytic local change of coordinates the 2-jet of a $\hz$ singularity can be written as:
\begin{equation}\label{fn_sing_2jet}
\left \{
\begin{aligned}
x'&=y-axz\\
y'&=-x-ayz\\
z'&=cz^2 +b(x^2+y^2).
\end{aligned}\right .
\end{equation}
At the 1-jet level we need to impose two degeneracy conditions which, joined to the open conditions $a\,b\,c\neq 0$, define a stratum of codimension
two in the space of germs of singularities of vector fields on $\mathbb{R}^3$.
Because $c\neq 0$, we can assume that this coefficient has been normalized by means of an appropriate scaling of coordinates.
Takens proved that there are six topological types, see Figure~\ref{fig_Singularitat}, but we are only interested in one of them,
that type characterized by the conditions
\begin{equation}\label{HZ*}
a>0, \quad b>0.
\end{equation}
$\hz$ singularities satisfying (\ref{HZ*}) are denoted by  $\hz^*$ in the sequel.
\begin{figure}[t!]
\begin{center}
\includegraphics[width=0.9\textwidth]{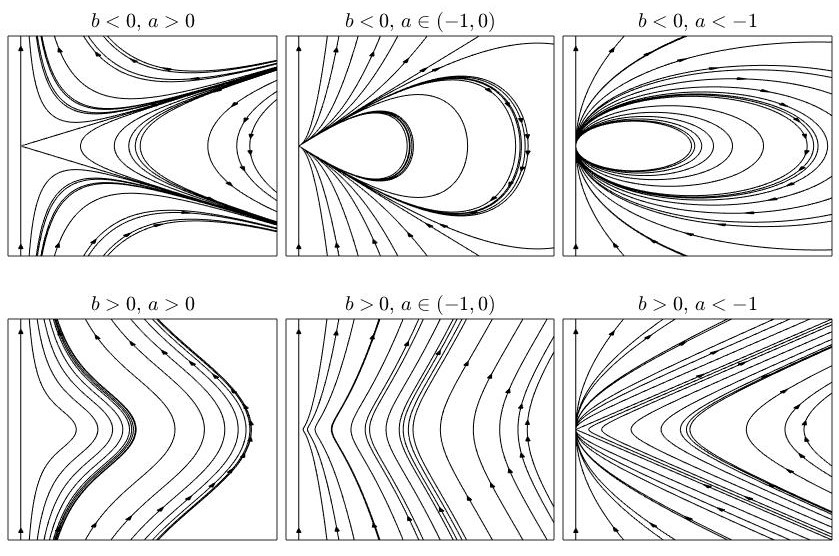}
\end{center}
\caption{Different topological types of Hopf-Zero singularities depending on $a$ and $b$.}
\label{fig_Singularitat}
\end{figure}

Generic unfoldings of $\hz$ singularities of codimension two were first studied by Guckenheimer \cite{Guc81} and Gavrilov \cite{Gav}.
The reader can find the bifurcation diagrams of suitable truncated normal forms for each of the six topological types in either \cite{GH90} or \cite{KUZ}.
Truncating at second order, any generic unfolding  can be written as:
\begin{equation}\label{fn_unf_2jet}
\left \{
\begin{aligned}
x'&=y+\nu x-axz\\
y'&=-x+\nu y -ayz\\
z'&=-\mu + z^2 +b(x^2+y^2)
\end{aligned}\right .
\end{equation}
with $\mu$, $\nu$ being the parameters of the unfolding, or, in cylindrical coordinates, as
\begin{equation}\label{fn_unf_2jet_cyl}
\left \{
\begin{aligned}
r'&=\nu r -arz\\
z'&=-\mu + z^2 +br^2\\
\theta'&=-1.
\end{aligned}\right.
\end{equation}

The characterization of $HZ^*$ singularities by means of their normal form does not allow to detect if a
given vector field belongs to $HZ^*$ without performing changes of variables to reduce it to its normal form.
For that reason, in Lemma \ref{lemma:normalform}, we prove that the intrinsic conditions:
\begin{equation}\label{opencondition1}
\begin{aligned}
&\left[\partial_{z^2}^2 \pi^z X^*\cdot \big (\partial_{x^2}^2 \pi^z X^*+ \partial_{y^2}^2 \pi^z X^*\big )\right](\mathbf{0})>0,\\
&\left[\partial_{z^2}^2 \pi^z X^*\cdot \big (\partial_{x,z}^2 \pi^x X^*+ \partial_{y,z}^2 \pi^y X^*\big )\right](\mathbf{0})<0 ,
\end{aligned}
\end{equation}
where $(\pi ^x,\pi^y,\pi ^z)$ denote the $(x,y,z)$ components of the vector field $X^*$,
define the vector fields $X^* \in HZ^*$ in terms of the two jet of $X^*$.
Furthermore, we prove that any unfolding $X_{\mu,\nu}$ satisfying the generic condition:
\begin{equation}\label{genericunfoldingdissipative}
\begin{aligned}
\Big[\Big(\partial_{x,\nu}^2 \pi^x X_{0,0} +  & \partial_{y,\nu} \pi^y X_{0,0}\Big)\partial_{\mu} \pi^z X_{0,0}
\\  & -\Big(\partial_{x,\mu}^2 \pi^x X_{0,0} + \partial_{y,\mu} \pi^y X_{0,0}\Big)\partial_{\nu} \pi^z X_{0,0} \Big](\mathbf{0})\neq 0
\end{aligned}
\end{equation}
has \eqref{fn_unf_2jet} as truncated normal form of order two.

When $\nu=0$, family (\ref{fn_unf_2jet_cyl}) has a first integral
$$
H(r,z)=r^{\frac{2}{a}}\left(-\mu+z^2+\frac{b}{1+a}r^2\right).
$$
Regarding the case of $\hz^*$ singularities, it easily follows from the existence of $H$  that, when $\mu>0$, the two-dimensional invariant manifolds of
the equilibrium points $p_\pm=(\pm\sqrt{\mu},0,0)$ form an invariant globe and, moreover,
the branches of the one-dimensional invariant manifolds contained inside the globe are also
coincident (see Figure \ref{fig_globe}).
\begin{figure}[t!]
\begin{center}
\includegraphics[width=0.4\textwidth]{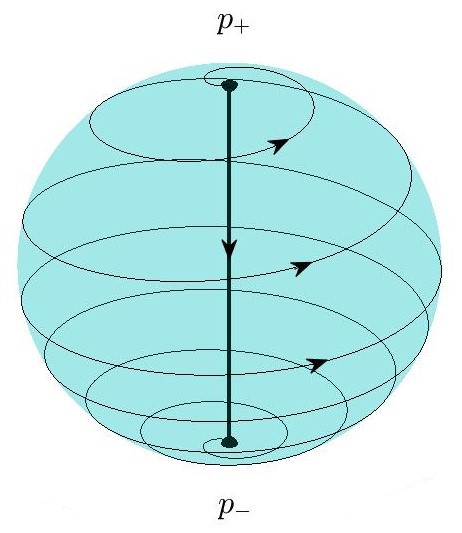}
\end{center}
\caption{Sketch of the invariant globe in the phase portrait of (\ref{fn_unf_2jet}) when $\nu=0$.
There are two saddle-focus points $p_+$ and $p_-$ such that $\mathrm{dim}(W^u(p_+))=\mathrm{dim}(W^s(p_-))=1$ and $\mathrm{dim}(W^s(p_+))=\mathrm{dim}(W^u(p_-))=2$.
One of the branches of $W^u(p_+)\setminus \{p_+\}$ coincides with one of the branches of $W^s(p_-)\setminus \{p_-\}$.
Moreover $W^u(p_-)\setminus \{p_-\}=W^s(p_+)\setminus \{p_+\}$.}
\label{fig_globe}
\end{figure}

Adding higher order terms, these invariant structures can be destroyed (even if the rotationally symmetry is preserved).
One can guess that the splitting of the invariant manifolds could lead to saddle-focus homoclinic connections.
The complexity to decide if such invariant manifolds whether split or not, strongly depends on the regularity of the
considered unfoldings, being the analytic case (the one considered in this paper) the more intricate one.

Simple geometrical arguments were supplied in \cite{Guc81,GH90} to support the possibility of the existence of Shilnikov homoclinic bifurcations in $\mathcal{C}^k$ unfoldings of
$\hz^*$ singularities.
For the specific case of families $X+\varepsilon Y$, where $X$ is the normal form (\ref{fn_unf_2jet}) and $\varepsilon>0$ is a perturbation parameter,
\cite{Gas93} proves the occurrence of Shilnikov homoclinic bifurcations near the codimension two point.
The tricky point here is that the generic unfoldings of a $\hz^*$ singularity can not  be written
in the form $X+\varepsilon Y$ (see also \cite{CK2004}).

The question is considered with a different perspective by Broer and Vegter in \cite{BroerV}.
They prove that for any generic $\mathcal{C}^\infty$ unfolding of a $\hz^*$ singularity there exists a ${\mathcal{C}}^\infty$ flat perturbation providing a family
exhibiting Shilnikov homoclinic bifurcations.
Moreover, abundance of unfoldings displaying the Shilnikov scenario is discussed, but using suitable topologies.
Because of the ``flat'' nature of the techniques, the remarkable results in \cite{BroerV} neither include any usable criterium for the existence of Shilnikov homoclinic bifurcations in  a given unfolding of $\hz^*$ singularities nor can be applied to analytic families unfolding
these singularities. Therefore the analytic case remains open since \cite{BroerV}.
As is made clear in that paper, Shilnikov phenomenon is beyond all orders.
Results in \cite{BroerV} were extended to time reversible unfoldings in \cite{LTW}.
In this paper we end with the whole discussion by showing that under generic hypotheses, any analytic unfolding of a Hopf-Zero singularity $\hz^*$ contains Shilnikov
homoclinic orbits.

General unfoldings of the $\hz^*$ singularity were considered in \cite{DIKS13}.
As argued there, introducing a scaling parameter $\varepsilon=\sqrt{\mu}$ and scaling variables and time,
one obtains either a singular perturbation problem with a pure rotation when $\varepsilon=0$
or a family with rotation speed tending to $\infty$ as $\varepsilon \to 0$.
In any of the two approaches there is no clear limit for the invariant manifolds of the two equilibrium points corresponding to the poles of the invariant globe
already discussed for (\ref{fn_unf_2jet_cyl}).
Nevertheless, one can apply the results in \cite{BF03,BF05} to prove that, when the scaling parameter tends to $0$, the invariant manifolds have a limit position,
given by the invariant manifolds of the equilibrium points at the $2$-jet level, at least when one considers restrictions to $z\geq 0$ or $z\leq 0$.
Therefore, for any generic unfolding of the $\hz^*$ singularity, splitting distance functions are well defined for both the one-dimensional and the
two-dimensional invariant manifolds. Using conjectured formulas for the splitting functions
and some extra conditions, existence of Shilnikov homoclinic bifurcation points is proven for general unfoldings.

On the other hand, in \cite{BCS13,BCS16a,BCS16b} explicit formulas for the splitting functions are obtained.
The splitting function for the one-dimensional invariant manifolds was achieved in \cite{BCS13}.
It follows that the distance between the one-dimensional invariant manifolds is exponentially small with respect to $\varepsilon$.
Moreover the coefficient in front of the dominant term depends on the full jet of the singularity.
The splitting function for the two-dimensional invariant manifolds was derived in \cite{BCS16a,BCS16b}.
The mean free terms in the asymptotic formula are exponentially small with respect to
$\varepsilon$ and, again, constants involved in the expression for those terms, which now depend on an angular variable, depend on the full jet of the singularity.
Because constants involved in the dominant terms depend on the full jet of the singularity, their computation can only be done by means of numerical
techniques (see \cite{DIKS13} for several examples).
In any case, such constants are the essential pieces to establish general criteria for the existence of Shilnikov homoclinic orbits.
Namely, depending on the accuracy of the result one deals with,
one needs to assume that either one or two of these constants do not vanish.

The point is that, putting together \cite{DIKS13,BCS13,BCS16a,BCS16b}
as well as additional results, we are able to provide general results for the existence of
Shilnikov homoclinic bifurcations in generic analytic unfoldings of the $\hz^*$ singularity.

At this point it should be noticed that $\hz$ singularities can be considered in two different contexts.
At the 1-jet level any $\hz$ singularity has divergence zero (note that the trace of (\ref{hopf_zero_linear}) is zero).
If one restricts to the set of volume-preserving systems, $\hz$ singularities are of codimension one, that is,
they occur generically in one parameter families of volume preserving vector fields.
In such a case we should consider $\nu=0$ in (\ref{fn_sing_2jet}) and (\ref{fn_unf_2jet_cyl}) and refer to unfoldings $X_{\mu}$.
We will require in this case, replacing (\ref{genericunfoldingdissipative}) the generic condition:
\begin{equation}\label{genericunfoldingconservative}
\partial_{\mu} \pi^z X_{0} (0,0,0) \neq 0 .
\end{equation}
In the general context, as already explained, $\hz$ singularities have codimension two, that is, they arise generically in two parameter families
$X_{\mu,\nu}$ of vector fields.
Note that according to \cite{broer81}, when working with families of vector fields with divergence zero,
the change of coordinates to reduce to normal form may be chosen volume preserving.

Now we are able to provide a qualitative version of our main result. Later, after introducing some technical details in Section~\ref{sec:pre}, we will give a more formal statement.

\begin{theorem}\label{qualitativeTheorem}
Let $X^*$ be a $\hz$ singularity satisfying the open conditions in~\eqref{opencondition1} and some generic conditions
(that will be given explicitly in theorems~\ref{Theoremdissipative} and~\ref{Theoremconservative}).
\begin{itemize}
 \item
In the volume preserving case, any divergence free generic unfolding $X_{\mu}$ satisfying the generic
condition~\eqref{genericunfoldingconservative}
has the following properties:
\begin{enumerate}
\item
For $\mu$ small enough, the vector field $X_{\mu}$ has,  at least, two heteroclinic orbits, formed by the intersection between the two-dimensional invariant manifolds of the equilibrium points.
\item
There exists a sequence of parameter values $\{\mu_n\}$
with $\mu_n\to 0$ as $n\to \infty$ such that the vector fields $X_{\mu_n}$ have a Shilnikov homoclinic orbit.
\end{enumerate}
\item In the general case, let $X^*$ satisfy the additional open condition
\begin{equation}\label{opencondition2}
\left |(\partial_{x,z}^2 \pi^x X^* + \partial_{y,z}^2 \pi^y X^*)(\mathbf{0})\right | < 2\left |\partial_{z^2}^2 \pi^z X^*)(\mathbf{0})\right|,
\end{equation}
and let $X_{\mu,\nu}$ be any analytic unfolding satisfying the generic condition~\eqref{genericunfoldingdissipative}.
Then, in the $(\mu,\nu)$ plane,
there exists an analytic  curve $\Gamma_0=\{(\mu,\nu)\, : \, \nu = \nu_0(\mu)\}$, with $\nu_0(0)=0$,
and domains $\mathcal{W}_2\subset \mathcal{W}_1$ contained in a wedge-shaped neighbourhood of
$\Gamma_0$ of a width that is exponentially small in $\sqrt{\mu}$ (see Figure~\ref{fig_wedge}) such that:
\begin{enumerate}
\item There exists an immersed curve $\cc \subset \mathcal{W}_1$ (maybe with more than one component)
such that the vector field $X_{\mu,\nu}$ has a Shilnikov homoclinic orbit, for $(\mu,\nu) \in \cc$.
\item
For $(\mu,\nu) \in \mathcal{W}_2$,
the vector field $X_{\mu,\nu}$ has, at least, two heteroclinic orbits, formed by the intersection between the two-dimensional invariant manifolds of the equilibrium points.
\item
There exists an open neighbourhood of the origin $\J\subset \R^3$ such that
$\mathcal{W}_2=\bigcup _{\rho \in \J} \Gamma_\rho $, where for any $\rho \in \J$,
$\Gamma_\rho=\{(\mu,\nu) \, : \, \nu = \nu_\rho(\mu)\}$ is a curve exponentially close to $\Gamma_0$.
In addition, for any $\rho\in \J$
there exists a sequence $\mu_n\to 0$ as $n\to \infty$
such that each vector field  $X_{\mu_n,\nu_\rho(\mu_n)}$ has a Shilnikov homoclinic orbit.
\end{enumerate}
\end{itemize}
\end{theorem}

\begin{figure}[t!]
\begin{center}
\includegraphics[width=0.5\textwidth]{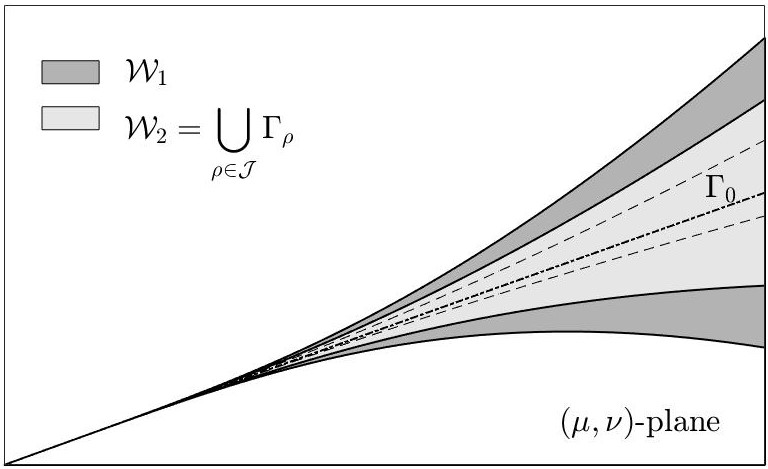}
\end{center}
\caption{Wedge-shaped domains $\mathcal{W}_2\subset \mathcal{W}_1$ and the curve $\Gamma_0$. $\mathcal{W}_1$ corresponds to the union of dark and light grey regions. The discontinuous lines in $\mathcal{W}_2$ are
two curves $\Gamma_\rho$.}
\label{fig_wedge}
\end{figure}

\begin{remark}
The generic conditions on the singularity $X^*$ in main Theorem~\ref{qualitativeTheorem} depend on the full jet of the singularity.
In fact they involve that some Stokes constants have to be different from zero.
\end{remark}

\begin{remark}\label{Rem0a2}
The additional open condition~\eqref{opencondition2} is not necessary to prove the existence of homoclinic orbits to saddle-focus equilibrium points,
but it appears in the main Theorem \ref{qualitativeTheorem} to get the expansivity condition which is required to have Shilnikov homoclinic orbits.
Namely, denoting the eigenvalues at the saddle-focus as $-\varrho \pm \omega i$ and $\lambda$, with $\varrho>0$ and $\lambda>0$,
the expansivity condition is $\lambda>\varrho$  (see \cite{Shil67,Shil65}). If~\eqref{opencondition2} holds, the eigenvalues at $p_+$
satisfy such condition for $\nu$ and $\mu$ small enough. At $p_-$ the expansivity condition is satisfied for the backward flow.
\end{remark}

\begin{remark}
It should be noticed that the occurrence of chaotic dynamics is also expectable in the unfolding of certain codimension-two Hopf-Hopf singularities, that is, singularities with two pairs of pure imaginary eigenvalues without resonances. The classification was obtained in \cite{Tak74} and, together with the study of truncated unfoldings, it can be found in \cite{GH90}. Similarly to the case of the Hopf-Zero singularities, the truncation of the normal form leads to planar reductions exhibiting, under appropriate assumptions, heteroclinic cycles involving equilibrium points and, bearing in mind the four-dimensional system, periodic orbits. The treatment required to get results of existence of homoclinic orbits is likely to be not far from the techniques described and used in this paper where we deal with the Hopf-Zero case.
\end{remark}

In Section~\ref{sec:pre}, to understand the role of all the open conditions stated in Theorem \ref{qualitativeTheorem},
we explain the derivation of the normal form of the unfoldings of a $\hz$ singularity up to second order.
Using such normal form, after appropriate scalings, we will write our unfoldings in the form to be used in the rest of the paper.
We will discuss the existence of equilibrium points, their invariant manifolds and their possible intersections.
At the end of Section~\ref{sec:pre} we will give a more quantitative version of our main theorem split into two results,
Theorem \ref{Theoremdissipative} in the general case and Theorem \ref{Theoremconservative} in the volume preserving case. Proofs are given in Section~\ref{sec:proof}.

\section{Preliminaries and a more quantitative result}\label{sec:pre}
This section is mainly devoted to recall some previous results about the dynamics of the unfoldings $X_{\mu,\nu}$ as well as
to write a more quantitative version of our main result, Theorem~\ref{qualitativeTheorem}.
In addition, we also prove Lemma~\ref{lemma:normalform} which gives the correspondence between the intrinsic conditions on
$X^*$ and $X_{\mu,\nu}$ stated in Theorem~\ref{qualitativeTheorem}
and the truncated normal form~\eqref{fn_unf_2jet}.

In the sequel, we will deal with both cases (the volume preserving and the general case) at the same time because
the volume preserving case is contained in the general one by putting  $\nu=0$ and assuming that $\text{tr}DX_{\mu,0}(0,0,0)=0$.

\subsection{Normal form and blow ups}
Even if the proof of the result below is elementary following the procedure indicated for instance in~\cite{Guc81,GH90,Tak74},
we  include it to follow the relation between the original vector field and the corresponding normal form.
\begin{lemma}\label{lemma:normalform}
Let $X^*$ be a $\hz$ singularity satisfying the open condition~\eqref{opencondition1}.
Assume the generic condition~\eqref{genericunfoldingdissipative} on the unfolding $X_{\mu,\nu}$ in the general case
and~\eqref{genericunfoldingconservative} in the volume preserving case.

Then there exists a real analytic change of variables\ and parameters such that the unfolding $X_{\mu,\nu}$ is given by
\begin{equation}\label{formanormalanalitica0}
\left \{
\begin{aligned}
x'&=y+\nu x-axz+f(x,y,z,\nu,\mu)\\
y'&=-x+\nu y -ayz+g(x,y,z,\nu,\mu) \\
z'&=-\mu+z^2 +b(x^2+y^2)+h(x,y,z,\nu,\mu)
\end{aligned}\right .
\end{equation}
where  $a,b>0$ and the functions $f,g,h$ are $\mathcal{O}(\|(x,y,z,\nu,\mu)\|^3)$ and real analytic.
In addition, if $X^*$ also satisfies~\eqref{opencondition2}, then $0<a<2$.
\end{lemma}
\begin{proof}
We prove the result in the general case, being the divergence free case a straightforward consequence by performing
the obvious changes.
The unfolding $X_{\mu,\nu}$ can be written as
$$
X_{\mu,\nu}(x,y,z)= (\alpha^* y, -\alpha^*x, 0)^\top + c \mu + d\nu + \O_{2}({x},{y},{z},\mu,\nu),
$$
with $c=(c_1,c_2,c_3)^\top \in \mathbb{R}^3$, $d=(d_1,d_2,d_3)^\top \in \mathbb{R}^3$ and where $\O_{2}$ stands for a vector field with components of order
$\O(\|({x},{y},{z},\mu,\nu)\|^2)$.
By means of the change of variables
$$
\bar{x} =x- \frac{(c_2 \mu + d_2\nu)}{\alpha^*}, \ \
\bar{y}=y+\frac{(c_1 \mu + d_1\nu)}{\alpha^*},
$$
we transform the unfolding into (keeping the same notation)
\begin{equation}\label{defabz}
X_{\mu,\nu}(x,y,z)= (\alpha^* y, -\alpha^*x, c_3 \mu + d_3\nu)^\top  + \O_{2}({x},{y},{z},\mu,\nu).
\end{equation}

After a second order normal form procedure applied to $X_{\mu,\nu}(x,y,z)$,
it can be written as (see~\cite{Guc81,GH90,Tak74}):
\begin{eqnarray*}
 \frac{d\tilde{x}}{d{t}}&=&\tilde{x}\left(\beta_0^1\nu + \beta_0^2 \mu -\beta_1\tilde{z}\right)+
 \tilde{y}\left(\alpha^*+\alpha_1\nu+\alpha_2\mu+\alpha_3\tilde{z}\right)+
 O_3(\tilde{x},\tilde{y},\tilde{z},\mu,\nu),\medskip\nonumber\\
\frac{d\tilde{y}}{d{t}}&=&-\tilde{x}\left(\alpha^*+\alpha_1\nu+\alpha_2\mu+
\alpha_3\tilde{z}\right)+\tilde{y}\left(\beta_0^1\nu+\beta_0^2 \mu-\beta_1\tilde{z}\right)
+O_3(\tilde{x},\tilde{y},\tilde{z},\mu,\nu),\medskip\\
\frac{d\tilde{z}}{d{t}}&=&-\gamma_0^1\mu - \gamma_0^2 \nu+\gamma_1\tilde{z}^2+
\gamma_2(\tilde{x}^2+\tilde{y}^2)\medskip\nonumber\\
&&+\gamma_3\mu^2+\gamma_4\nu^2+\gamma_5\mu\nu +\gamma_{1}^{1}\mu \tilde z+\gamma_{1}^{2}\nu \tilde z+
O_3(\tilde{x},\tilde{y},\tilde{z},\mu,\nu).\nonumber
\end{eqnarray*}
The coefficients $\beta_1,\gamma_1,\gamma_2$ and $\alpha_3$ depend only on the initial singularity and straightforward computations give:
$$
\beta_1 = -\frac{1}{2}\big (\partial_{x,z}^2 \pi^x X^* + \partial_{y,z}^2 \pi^y X^*)(\mathbf{0}),\qquad
\gamma_1 =\frac{1}{2}\partial_{z^2}^2 \pi^z X^* (\mathbf{0}),
$$
$$
\gamma_2= \frac{1}{4}\left(\partial_{x^2}^2 \pi^z X^*+ \partial_{y^2}^2 \pi^z X^*\right)(\mathbf{0}).
$$
It follows from  conditions in~\eqref{opencondition1} and~\eqref{opencondition2} that
\begin{equation}\label{openconditionparameters1}
\gamma_1 \gamma_2 >0,\qquad \gamma_1 \beta_1>0,\qquad |\beta_1|<2|\gamma_1|.
\end{equation}
\\
The terms $\beta_0^1,\beta_0^2, \gamma_0^1,\gamma_0^2$ depend on the second order derivatives of $X_{\mu,\nu}$ at zero,
that is, on the terms of degree two of the unfolding.
Namely, it is straightforward to check that
\begin{eqnarray*}
\beta_0^1 = \frac{1}{2} \big (\partial_{x,\nu}^2 \pi^x X_{\mu,\nu} + \partial_{y,\nu} \pi^y X_{\mu,\nu}\big ) (\mathbf{0}) && \gamma_0^1= -c_3,
\\
\beta_0^2 = \frac{1}{2} \big (\partial_{x,\mu}^2 \pi^x X_{\mu,\nu} + \partial_{y,\mu} \pi^y X_{\mu,\nu}\big )(\mathbf{0}) && \gamma_{0}^2=-d_3,
\end{eqnarray*}
where $c_3$ and $d_3$ are introduced in~\eqref{defabz}.
Note that
$$
\gamma_0^1= -c_3=-\partial_{\mu} \pi^z X_{0,0} (\mathbf{0}) \qquad  \gamma_0^2= -d_3=-\partial_{\nu} \pi^z X_{0,0} (\mathbf{0}).
$$
It follows from the generic condition~\eqref{genericunfoldingdissipative} that $\beta_0^1\gamma_0^1-\beta_0^2\gamma_0^2\neq 0$.
Hence we can introduce new parameters
\begin{equation*}
\tilde \nu = \beta_0^1\nu + \beta_0^2 \mu \, ,
\qquad
\tilde \mu = \gamma_0^1\mu + \gamma_0^2 \nu \, ,
\end{equation*}
obtaining
\begin{eqnarray}\label{normalformcartabc}
\frac{d\tilde{x}}{d{t}}&=&\tilde{x}\left(\tilde\nu -\beta_1\tilde{z}\right)+
 \tilde{y}\left(\alpha^*+\tilde\alpha_1\tilde \nu+\tilde \alpha_2 \tilde \mu+\alpha_3\tilde{z}\right)+
 O_3(\tilde{x},\tilde{y},\tilde{z},\tilde \mu, \tilde \nu),\medskip\nonumber\\
\frac{d\tilde{y}}{d{t}}&=&-\tilde{x}\left(\alpha^*+\tilde\alpha_1\tilde \nu+\tilde \alpha_2 \tilde \mu+\alpha_3\tilde{z}\right)+
\tilde{y}\left(\tilde\nu-\beta_1\tilde{z}\right)
+O_3(\tilde{x},\tilde{y},\tilde{z},\tilde\mu,\tilde\nu),\medskip\\
\frac{d\tilde{z}}{d{t}}&=&-\tilde\mu+\gamma_1\tilde{z}^2+
\gamma_2(\tilde{x}^2+\tilde{y}^2)
+\tilde\gamma_3\tilde\mu^2+\tilde\gamma_4\tilde\nu^2+\tilde\gamma_5\tilde\mu\tilde\nu
\medskip\nonumber\\
&&
+\tilde\gamma_{1}^{1}\tilde\mu z+\tilde\gamma_{1}^{2}\tilde\nu z+
O_3(\tilde{x},\tilde{y},\tilde{z},\tilde\mu,\tilde\nu).\nonumber
\end{eqnarray}
Expressions of $\tilde\alpha_1$, $\tilde\alpha_2$, $\tilde\gamma_3$, $\tilde\gamma_4$, $\tilde\gamma_5$, $\tilde\gamma_1^1$ and
$\tilde\gamma_1^2$ are not provided because they are not relevant in the sequel.
Since $\alpha^{*}>0$, in a neighbourhood of $(\tilde{x},\tilde{y},\tilde{z},\tilde{\mu},\tilde{\nu})=(0,0,0,0,0)$ we can multiply the vector field
(\ref{normalformcartabc}) by the function
$$
\frac{1}{\alpha^*+\tilde \alpha_1 \tilde{\nu}+\tilde \alpha_2\tilde{\mu}+\alpha_3 \tilde z}
=\frac{1}{\alpha^*}-\frac{\tilde \alpha_1}{(\alpha^*)^2}\tilde{\nu}-\frac{\tilde \alpha_2}{(\alpha^*)^2}\tilde{\mu}-\frac{\alpha_3}{(\alpha^*)^2}\tilde z
+\O(\|(\tilde z,\tilde{\mu},\tilde{\nu})\|^2)
$$
to get the equivalent family
\begin{eqnarray*}
\frac{d \bar x}{d \tau}&=&\bar x\left(\bar\beta_0 \tilde{\nu} -\bar\beta_1 \bar z\right)+\bar y+\O(\|(\bar x,\bar y,\bar z,\tilde{\mu},\tilde{\nu})\|^3)
\nonumber
\\
\frac{d \bar y}{d \tau}&=&-\bar x+\bar y\left(\bar\beta_0 \tilde{\nu} -\bar\beta_1 \bar z\right)+\O(\|(\bar x,\bar y,\bar z,\tilde{\mu},\tilde{\nu})\|^3)
\\
\frac{d \bar z}{d \tau}&=&-\bar\gamma_0\tilde{\mu}+\bar\gamma_1 \bar z^2+\bar\gamma_2(\bar x^2+\bar y^2)+\bar\gamma_3\tilde{\mu}^2
+\bar\gamma_4\tilde{\nu}^2+\bar\gamma_5\tilde{\mu}\tilde{\nu}
\nonumber
\\
&&+\bar\gamma_{1}^{1}\tilde\mu z+\bar\gamma_{1}^{2}\tilde\nu z
+\O(\|(\bar x,\bar y,\bar z,\tilde{\mu},\tilde{\nu})\|^3)
\nonumber
\end{eqnarray*}
with
$$
\bar\beta_0=\frac{1}{\alpha^*},\qquad
\bar\beta_1=\frac{\beta_1}{\alpha^*},\qquad
\bar\gamma_0=\frac{1}{\alpha^*},\qquad
\bar\gamma_1=\frac{\gamma_1}{\alpha^*},\qquad
\bar\gamma_2=\frac{\gamma_2}{\alpha^*}.
$$
Precise expressions for $\bar\gamma_3$, $\bar\gamma_4$, $\bar\gamma_5$, $\bar\gamma_1^1$ and $\bar\gamma_1^1$ are not relevant for further calculations.
Then, since the condition~\eqref{openconditionparameters1} is satisfied, $\bar \gamma_1\neq 0$ and we can introduce the
change of variables:
$$
\hat x=\bar x, \qquad \hat y=\bar y ,\qquad \hat z=\bar \gamma_1\bar z+\frac{\bar \gamma_1^1\tilde\mu+\bar \gamma_1^2\tilde\nu}{2},
$$
and parameters:
\begin{align*}
\hat\mu & =\tilde\gamma_0\tilde\gamma_1 \tilde\mu+\frac{(\bar\gamma_1^1 \tilde \mu+\bar\gamma_1^2 \tilde \nu)^2 }{4}-
\bar\gamma_1\bar\gamma_3\tilde\mu^2-\bar\gamma_1\bar\gamma_4\tilde\nu^2-\bar\gamma_1\bar\gamma_5\tilde \mu \tilde \nu,
\\
\hat\nu&=\bar\beta_0 \tilde \nu+ \frac{\bar\beta_1(\bar\gamma_1^1 \tilde \mu+\bar\gamma_1^2 \tilde \nu)}{2\bar\gamma_1}\tilde \mu,
\end{align*}
to obtain the normal form \eqref{formanormalanalitica0}:
\begin{eqnarray*}
\frac{d \hat x}{d \tau}&=&\hat x\left(\hat\nu -a \hat z\right)+\hat y+O(\|(\hat x,\hat y,\hat z,\hat\mu,\hat\nu)\|^3)
\nonumber
\\
\frac{d \hat y}{d \tau}&=&-\hat x+\hat y\left(\hat\nu -a \hat z\right)+O(\|(\hat x,\hat y,\hat z,\hat\mu,\hat\nu)\|^3)
\\
\frac{d \hat z}{d \tau}&=&-\hat\mu+\hat z^2+b(\hat x^2+\hat y^2)+O(\|(\hat x,\hat y,\hat z,\hat\mu,\hat\nu)\|^3)
\nonumber
\end{eqnarray*}
with
$$
a=\frac{\bar\beta_1}{\bar\gamma_1},\qquad b=\bar\gamma_1\bar\gamma_2
$$
or equivalently:
$$
a = \frac{\beta_1}{\gamma_1},\qquad b= \frac{\gamma_1 \gamma_2}{(\alpha^* )^2}.
$$

Note that the condition~\eqref{openconditionparameters1} is now equivalent to:
$$
b>0, \qquad a>0, \qquad a<2,
$$
where the two first conditions correspond to~\eqref{opencondition1} and the third one to~\eqref{opencondition2}.
\end{proof}

To conclude this preliminary section, we introduce the new parameters $ \varepsilon, \sigma$ as:
\begin{equation*}
(\mu,\nu) = (\varepsilon^2 ,\varepsilon \sigma),\qquad \varepsilon >0
\end{equation*}
and the blow up
$$
x=\varepsilon \bar{x} ,\quad y=\varepsilon \bar{y},\quad z=\varepsilon \bar{z}.
$$
Then, system~\eqref{formanormalanalitica0} becomes (dropping bars of the notation):
\begin{equation}\label{formanormalanalitica1}
\begin{aligned}
\frac{d {x}}{dt}&= \varepsilon {x}(\sigma {x}- a{z}) + {y}+
\varepsilon^{-1} f(\varepsilon {x},\varepsilon {y},\varepsilon {z},\varepsilon^2,\varepsilon \sigma)\\
\frac{d{y}}{dt}&=-{x}+\varepsilon {y}(\sigma  -a {z})+
\varepsilon^{-1}g(\varepsilon {x},\varepsilon {y},\varepsilon {z},\varepsilon^2,\varepsilon \sigma) \\
\frac{d{z}}{dt}&=-\varepsilon+ \varepsilon {z}^2 + \varepsilon b({x}^2+{y}^2)+
\varepsilon^{-1}h(\varepsilon {x},\varepsilon {y},\varepsilon {z},\varepsilon^2,\varepsilon \sigma).
\end{aligned}
\end{equation}
Also, scaling time, $s=\varepsilon t$ we obtain the following vector field, that we will call, abusing notation, $X_{\varepsilon, \sigma}$
in the general case and $X_{\varepsilon}$ in the volume preserving case ($\sigma=0$).
\begin{equation}\label{formanormalanalitica2}
\begin{aligned}
\frac{d{x}}{d s}&= {x} (\sigma - a{z}) + \frac{{y}}{\varepsilon}+
\varepsilon^{-2} f(\varepsilon {x},\varepsilon {y},\varepsilon {z},\varepsilon^2,\varepsilon \sigma)\\
\frac{d{y}}{d s}&=-\frac{{x}}{\varepsilon}+ {y} (\sigma - a{z})+
\varepsilon^{-2} g(\varepsilon {x},\varepsilon {y},\varepsilon {z},\varepsilon^2,\varepsilon \sigma) \\
\frac{d{z}}{d s}&= -1+ {z}^2   +b({x}^2+{y}^2)+
\varepsilon^{-2} h(\varepsilon {x},\varepsilon {y},\varepsilon {z},\varepsilon^2,\varepsilon \sigma).
\end{aligned}
\end{equation}
From now on we will work with the generic unfoldings already in the form~\eqref{formanormalanalitica2}.
\begin{remark}
Families (\ref{formanormalanalitica1}) and (\ref{formanormalanalitica2}) were also used in \cite{DIKS13,BCS13,BCS16a,BCS16b}.
Note that for family (\ref{formanormalanalitica1}), the limit when $\varepsilon\to 0$ is merely a rotation around the vertical axis.
On the contrary, in the case of (\ref{formanormalanalitica2}) there is no regular limit when $\varepsilon\to 0$ because the rotation speed tends to $\infty$.
\end{remark}

\subsection{Previous results and main theorems}

We will state and prove the quantitative version of Theorem~\ref{qualitativeTheorem} in terms of system~\eqref{formanormalanalitica2}.
For that reason, the first result we use is the following lemma, whose proof can be found  in \cite{BaSe06},
which assures the existence of saddle-focus equilibrium points:
\begin{lemma}
Consider system \eqref{formanormalanalitica2} with $b>0$, $ a>0$ and $|\sigma| <a$.
Then, there exists $\varepsilon _0>0$ such that, for $0<\varepsilon<\varepsilon_0$, the vector field has two equilibrium points
$p_\pm$ (depending on $\varepsilon$) of saddle-focus type such that
$p_+$ ($p_-$) has a one-dimensional unstable (stable) manifold and a two-dimensional stable (unstable) one.
\end{lemma}
Observe that, when $f=g=h=0$ in \eqref{formanormalanalitica2}, the points $p_\pm$ are $(0,0,\pm 1)$ and they are connected by the
heteroclinic orbit:
\begin{equation}\label{W1}
W_1=\{ x=y=0,\; |z| < 1\}
\end{equation}
which consists on a branch of the $1$-dimensional unstable manifold of $p_+=(0,0,1)$ that coincides with a branch of the $1$-dimensional
stable manifold of $p_-=(0,0,-1)$.

However, one expects that for generic $f,g,h$, the one dimensional heteroclinic connection \eqref{W1} breaks.
Next theorem,  which corresponds to Theorem 1 in~\cite{BCS13} just adapting the notation, gives the distance
$S^1(\sigma, \varepsilon)$
between the $1$-dimensional invariant manifolds  of $p_+$ and $p_-$ of system \eqref{formanormalanalitica2},
which is  different from zero under generic conditions on the singularity $X^*$, see Figure~\ref{fig:splitting}.

\begin{figure}
\centering
\subfloat{\includegraphics[height=5cm]{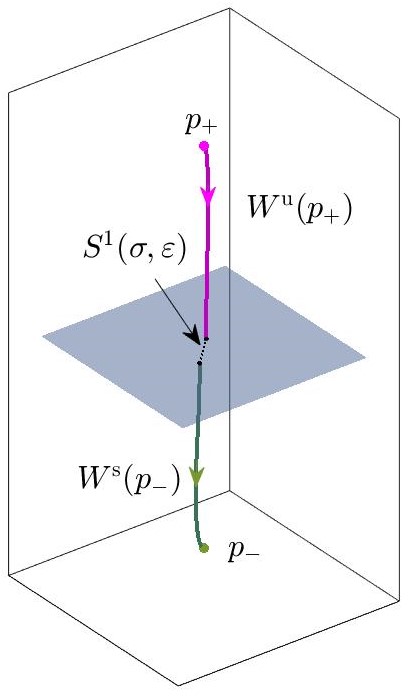}}
\hspace{1.5cm}
\subfloat{\includegraphics[height=5cm]{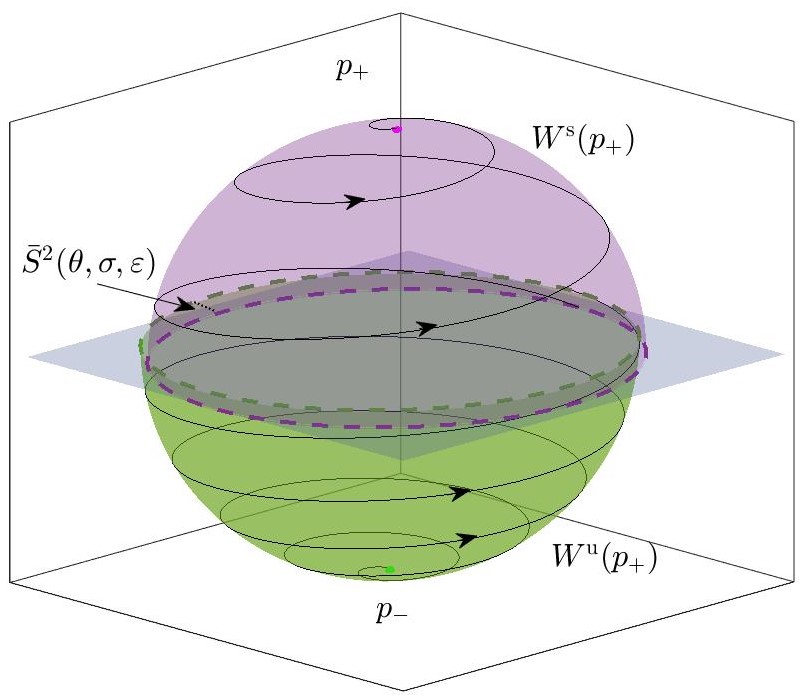}}
\caption{On the left, the splitting of the one dimensional heteroclinic connection $W^1$. On the right, the corresponding
breakdown of the two dimensional heteroclinic connection.}
\label{fig:splitting}
\end{figure}
\begin{theorem}[\cite{BCS13}]\label{thm:splitting1d}
Consider system~\eqref{formanormalanalitica2} with $a,b>0$ and $|\sigma|<a$.
Then there exist $\varepsilon_0>0$ and a real constant $\mathcal{C}^*$, such that, for $0<\varepsilon\le \varepsilon_0$,
the distance ${S}^{1}(\varepsilon,\sigma)$ between the one-dimensional stable manifold of
$p_{-}$ and the one-dimensional unstable manifold of $p_{+}$ when they meet the plane $z=0$ is given by
$$
{S}^{1}(\varepsilon,\sigma)=\varepsilon^{-1+a} \mathrm{e}^{-\frac{\pi}{2\varepsilon}}
\mathrm{e}^{-\frac{\pi c_0}{2}} \left (\mathcal{C}^* + \mathcal{O}(|\log \varepsilon|^{-1})\right )
$$
with $c_0=\lim_{z\to 0} z^{-3} h(0,0,z,0,0)$.

The constant $\mathcal{C}^*$ depends on the full jet of the singularity $X^*$ and is different from zero for generic singularities.
\end{theorem}

The study of the $2$-dimensional invariant manifolds of the equilibrium points $p_\pm$  of  system~\eqref{formanormalanalitica2} is more involved,
see \cite{BCS16a,BCS16b} for a detailed study of the relative position of these manifolds. We give some details below.

When $f=g=h=0$ and $\sigma=0$ one can see that the $2$-dimensional manifolds coincide forming an ellipsoid given by:
\begin{equation}\label{eq:sfere}
z^2+\frac{b}{a+1}(x^2+y^2)=1,
\end{equation}
but in the general case ($\sigma\neq 0$) it is possible that the $2$-dimensional unstable manifold of $p_-$ and the
$2$-dimensional stable manifold of $p_+$ do not intersect.
Indeed, when the parameter $\sigma$ is not of order $\varepsilon$, the position of the $2$-dimensional manifolds is already known;
they do not intersect and the distance between them is of order $\sigma$.
This fact is easily obtained by doing another step of the normal form procedure to system~\eqref{formanormalanalitica2}
and studying the position of these manifolds in the normal form of order three.
Moreover, in this case, the  existence of Shilnikov homoclinic bifurcations is not possible.
Indeed, when the $2$-dimensional invariant manifolds do not intersect, there exist either forward or backward trapping
regions which prevent the existence of homoclinic connections (see details in \cite{DIKS13}).

For this reason, from now on, we restrict our parameters and we will take $\sigma= \mathcal{O}(\varepsilon)$.
The result that we will use in our case is already done in~\cite{BCS16b} (and also~\cite{BCS16a}).
In that work, symplectic polar coordinates were considered:
$$
x = \sqrt{2r} \cos \theta,\qquad y=\sqrt{2r} \sin \theta,
$$
and it was proven that the  $2$-dimensional invariant manifolds of $p_{\pm}$ can be parameterized
by $\theta$ and $z=\tanh (a u)$ as
$$
r=r^{\uns,\sta}(u,\theta) = \frac{a+1}{2b \cosh^2 (au)} + r^{\uns,\sta}_1(u,\theta),\qquad r^{\uns,\sta}_1(u,\theta)=\O(\varepsilon), \quad u\in [-T,T].
$$
In particular, the intersection of the two dimensional invariant manifolds of $p_{\pm}$  with the plane $z=0$
are two closed curves $\mathbf{C}^{\uns,\sta}$ that can be parameterized as:
\begin{equation}\label{eq:Cus}
\mathbf{C}^{\uns,\sta}=\{(x,y,0), \ x=\sqrt{2r^{\uns,\sta}(0,\theta)}\cos \theta, \ y=\sqrt{2r^{\uns,\sta}(0,\theta)}\sin \theta , \ \theta \in \T \}.
\end{equation}

Let us call
\begin{equation}\label{disttwodimensional}
\bar{S}^{2} (\theta,\varepsilon,\sigma) = r^{\uns}(0,\theta)-r^{\sta}(0,\theta),
\end{equation}
the radial symplectic distance between the manifolds when they meet the plane $z=0$. From Theorem 2.16 of \cite{BCS16b} one has:

\begin{theorem}[\cite{BCS16b}] \label{thm:heteroclinic2d}
Consider the radial symplectic distance $\bar{S}^{2}(\theta,\varepsilon,\sigma)$ defined
in~\eqref{disttwodimensional}.
There exist $\varepsilon_0>0$, $\sigma_0>0$ and  a complex constant $\mathcal{C}_1^*+i\mathcal{C}_2^*$,  such that for $|\sigma|\le \sigma_0 \varepsilon$
and $0<\varepsilon\le \varepsilon_0$, one has
\begin{eqnarray*}
\bar {S}^{2}(\theta,\varepsilon,\sigma) &=& \Upsilon^{[0]}(1+\O(\varepsilon))+\varepsilon^{-2-\frac{2}{a}} \mathrm{e}^{-\frac{\pi}{2a\varepsilon}}
\left [\mathcal{C}_1^* \cos (\theta - a^{-1} L_0 \log \varepsilon)\right.
\\
&&\left.+\mathcal{C}_2^* \sin (\theta - a^{-1} L_0 \log \varepsilon) + \mathcal{O}(|\log \varepsilon|^{-1})\right ]
\end{eqnarray*}
where $L_0$ is a constant depending on the terms of degree three of the singularity $X^*$ (see Remark 5.17 in
\cite{BCS16a} for details) and
$$
\Upsilon^{[0]}=\Upsilon^{[0]}(\varepsilon,\sigma)= \sigma I+\varepsilon J +\mathcal {O}_2(\varepsilon,\sigma)
$$
with $I\ne 0$ and $J$ computable constants (for details, see formulas (90) and (91) of \cite{BCS16a}).

In addition, there exists a curve
$$
\Gamma_0^*=\{(\varepsilon,\sigma), \ \sigma=\sigma^*_0(\d)=-\frac{J}{I}\d+\mathcal{O}(\d^{2})\}
$$
such that for all $0\leq\d\leq\d_0$ one has:
$$
\Upsilon^{[0]}=\Upsilon^{[0]}(\d,\sigma^*_0(\d))=0.
$$
Moreover, given any constants $c_1$, $c_2$ and $c_3>0$, there exists a curve
\begin{equation*}
\Gamma^* _\rho=\{(\varepsilon,\sigma), \ \sigma=\sigma^*_{\rho}(\d)=\sigma^*_0(\d)+\mathcal{O}(\d^{c_2}\mathrm{e}^{-\frac{c_3\pi}{2a\d}})\}
\end{equation*}
with $\rho=(c_1,c_2,c_3)$,
such that for all $0\leq\d\leq\d_0$ one has:
$$
\Upsilon^{[0]}=\Upsilon^{[0]}(\d,\sigma^*_{\rho}(\d))=c_1\d^{c_2}\mathrm{e}^{-\frac{c_3\pi}{2a\d}}.
$$

In the volume preserving case $\Upsilon^{[0]}=0$.

In addition, the constant $\mathcal{C}_1^*+i\mathcal{C}_2^*\neq 0$ for generic singularities $X^*$ and depends on the full jet of $X^*$.
\end{theorem}

We emphasize that, since the unfoldings are analytic, the distance between the invariant manifolds is exponentially small for adequate values of the parameters.
This fact is intrinsic to the analytic case and lies in the field of singular perturbation theory.
Conversely, when we are in the finitely many differentiable case, the classical perturbation theory can be applied to compute the distances
$S^1$ and $\bar{S}^2$. They both will be, generically, $\mathcal{O}(\varepsilon^k)$ for some $k>0$, for values of $\sigma$ in an adequate curve.

Using the notation we have already introduced, we can write a more quantitative version of Theorem \ref{qualitativeTheorem} for the general case:

\begin{theorem}\label{Theoremdissipative} [General case]
Consider system \eqref{formanormalanalitica2} and assume
the open conditions $0<a<2 $ and  $b>0$ and the generic condition $\mathcal{C}^* \ne 0$, where
$\mathcal{C}^*$ is given in Theorem \ref{thm:splitting1d},
on the initial singularity $X^*$.

\begin{itemize}
\item
In the $(\varepsilon,\sigma)$ plane, with $\varepsilon$ small enough, there exists an immerse curve $\cc^*$, which lies in a wedge-shaped neighbourhood
$\mathcal{W}_1^*$ of the curve
$\Gamma_0^*$ given in Theorem \ref{thm:heteroclinic2d} of a width
that is at most $2c_1\d ^{-2-\frac{2}{a}}\mathrm{e}^{-\frac{\pi}{2a \varepsilon}}$, with either $c_1=\sqrt{(\mathcal{C}_1^*)^2+(\mathcal{C}_2^*)^2}$
if $\mathcal{C}_1^* + i \mathcal{C}_2^* \neq 0$ or $c_1>0$ if $\mathcal{C}_1^* + i \mathcal{C}_2^* = 0$,
such that for $(\varepsilon,\sigma) \in \cc^*$, any analytic unfolding
$X_{\varepsilon,\sigma}$ has a Shilnikov homoclinic orbit to the equilibrium point $p_{+}$.
\item
Assume moreover the generic condition on the singularity $X^*$: $\mathcal{C}_1^*\ne 0$ or $\mathcal{C}_2^*\ne 0$,
where $\mathcal{C}_1^*+i\mathcal{C}_2^*$  is given in Theorem~\ref{thm:heteroclinic2d}.
Let $0<\kappa<1$ be any constant.
Take any curve $\Gamma^* _\rho$, $\rho=(c_1,c_2,c_3)$ as given in  Theorem~\ref{thm:heteroclinic2d} with
$$|c_1|\leq \kappa \sqrt{(\mathcal{C}_1^*)^2+(\mathcal{C}_2^*)^2},\qquad c_2\ge -2-\frac{2}{a}, \qquad c_3\ge 1$$
and $\varepsilon$ small enough.
Let ${\displaystyle{\mathcal{W}_2^* =\bigcup_{\rho} \Gamma_\rho^*}}$.
Then
\begin{enumerate}
\item
For any $(\varepsilon,\sigma)\in \mathcal{W}_2^*$ there are, at least, two heteroclinic orbits from $p_-$ to $p_+$.
\item
For any curve $\Gamma^* _\rho \subset \mathcal{W}_2^*$, there exists a sequence of parameter values $\{\varepsilon_n\}$,
with $\varepsilon_n\to 0$ as $n\to \infty$ and  $(\varepsilon_n, \sigma^*_{\rho}(\varepsilon_n))\in \mathcal{W}_2^*$,
such that the vector field $X_{\varepsilon_n, \sigma^*_{\rho}(\varepsilon_n)}$ has
a Shilnikov homoclinic orbit to the equilibrium point $p_{+}$.
\end{enumerate}
\end{itemize}
\end{theorem}

The precise results in the volume preserving case are collected in the following result:
\begin{theorem}\label{Theoremconservative} [Volume preserving case]
Consider system~\eqref{formanormalanalitica2} and assume that it is volume preserving.
In particular, $a=1$ and $\nu=0$.
Assume moreover the open condition  $b>0$ and the generic conditions $\mathcal{C}^* \ne 0$,  $(\mathcal{C}_1^*)^2+(\mathcal{C}_2^*)^2\ne 0$ on the initial singularity $X^*$.
\begin{enumerate}
\item
Then, for $\varepsilon$ small enough there are, at least, two heteroclinic orbits from $p_-$ to $p_+$.
\item
There exists a sequence of parameter values $\{\varepsilon_n\}$,
with $\varepsilon_n >0$ and $\varepsilon_n\to 0$ as $n\to \infty$ such that the
vector field  $X_{\varepsilon_n}$ has a Shilnikov homoclinic orbit to the equilibrium point  $p_{+}$.
\end{enumerate}
\end{theorem}

Observe that  Theorem \ref{qualitativeTheorem} is just a corollary of theorems~\ref{Theoremdissipative} and~\ref{Theoremconservative}
undoing the changes of variables from System~\eqref{formanormalanalitica2} to System~\eqref{formanormalanalitica0}.
\begin{remark}
Recall that condition $0<a<2$ is equivalent to~\eqref{opencondition2}.
As we pointed out in Remark~\ref{Rem0a2}, to prove the existence of homoclinic orbits to saddle-focus equilibrium points we only need to assume
$a>0$. However, the condition $0<a<2$ is necessary to guarantee that such homoclinic orbits are of Shilnikov type.
\end{remark}

\section{Proof of Theorems \ref{Theoremdissipative} and \ref{Theoremconservative}}\label{sec:proof}

The proof of these two theorems will have some common setting: to prove the existence of homoclinic orbits to the north point
$p_+$ we need to control its global $1$-dimensional unstable manifold $W^\uns(p_+)$
and more concretely its intersections with the plane $z=0$.

Theorem \ref{thm:splitting1d} gives us information about the ``first'' intersection of $W^\uns(p_+)$ with the plane $\{z=0\}$ and the distance
$S^1(\d,\sigma)$,
between $W^\uns(p_+)$ and the $1$-dimensional stable manifold of the south point $p_-$; $W^\sta(p_-)$ (see Figure~\ref{fig:splitting}).
Using the precise results in Theorems~\ref{thm:splitting1d} and~\ref{thm:heteroclinic2d} we will prove that $W^\uns(p_+)$
intersects again the plane $z=0$ in a point $Q_0$ which is ``close enough'' to the $2$-dimensional manifold $W^\uns(p_-)$.
The goal is to get a sharp bound for the distance between $Q_0$ and the curve $\mathbf{C}^\uns=W^\uns(p_-)\cap \{z=0\}$ (see Figure \ref{fig_segundocorte}).

\begin{figure}[t!]
\subfloat{
\includegraphics[width=0.45\textwidth]{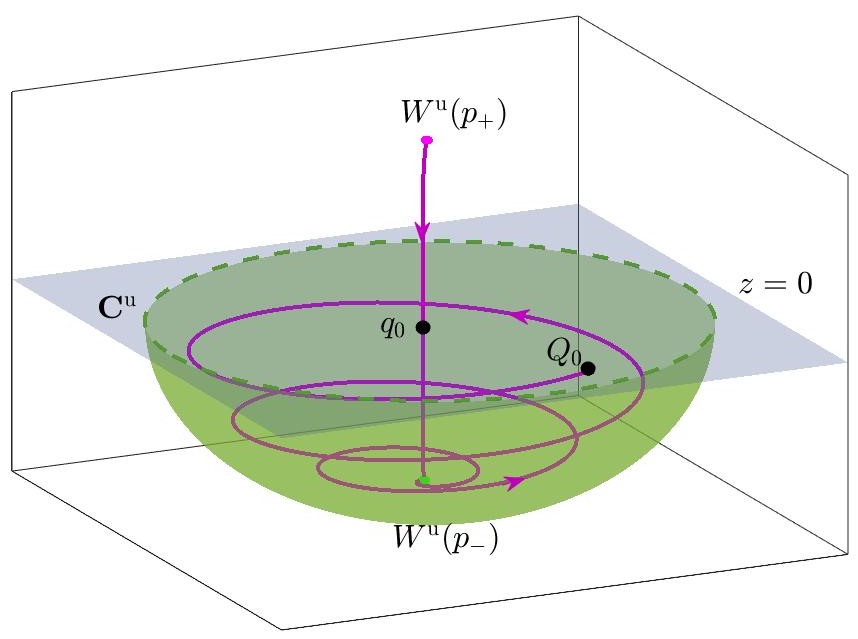}}
\hspace{0.8cm}
\subfloat{
\includegraphics[width=0.45\textwidth]{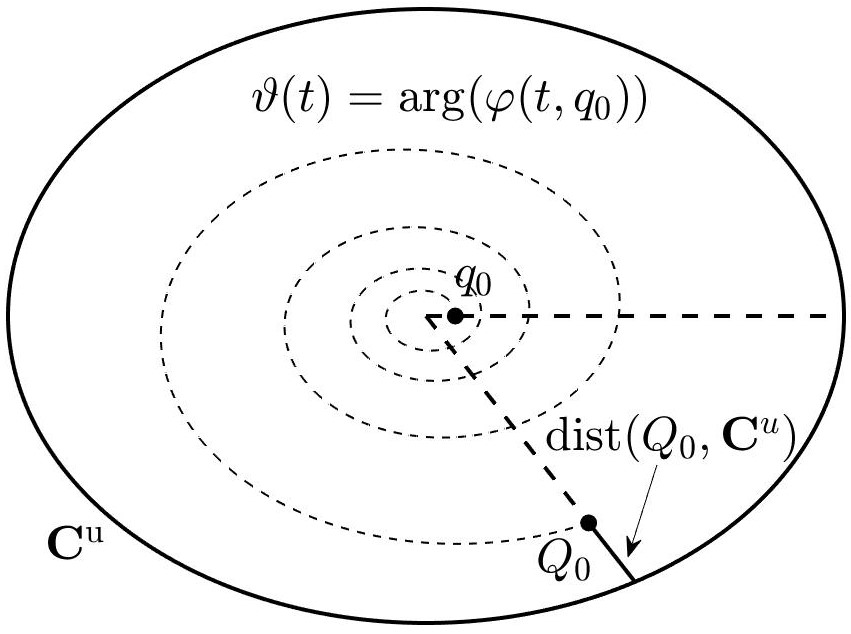}
}
\caption{The points $q_0$ and $Q_0$ correspond to the first and second intersections of $W^\uns(p_+)$ with $z=0$.
The curve $\mathbf{C}^\uns$ is given by the first intersection of $W^\uns(p_-)$ with $z=0$.
We need to get appropriate bounds for the angle $\theta_0=\vartheta(t)$ and for the distance between $Q_0$ and $\mathbf{C}^\uns$.
On the left the three dimensional figure, on the right the plane $\{z=0\}$.
}
\label{fig_segundocorte}
\end{figure}

More explicitly, we will prove the following key result:
\begin{proposition}\label{prop:2crossings}
Consider System~\eqref{formanormalanalitica2} and let $\varphi(t;q)$ be its flow. Then
$W^{\uns}(p_+)$ cross the plane $z=0$ at least twice.
Let us  call $q_0$ and $Q_0$ the first and the second intersection points.
Let
$\tau>0$ be such that $Q_0=\varphi(\tau;q_0)$. Then we have:
\begin{enumerate}
\item
The distance between $Q_0$ and the unstable curve $\mathbf{C}^\uns$ given in \eqref{eq:Cus} is bounded from above by
$$
\mathrm{dist}(Q_0,\mathbf{C}^\uns)\leq \bar C_1 \varepsilon ^{2-\frac{2}{a}}\mathrm{e}^{-\frac{\pi}{a \varepsilon}},
$$
for some constant $\bar C_1 >0$.
\item
Consider the $\mathcal{C}^\infty$ function $\vartheta(t)$ giving the argument of
$\varphi(t;q_0)$ and define 
$\theta_0=\vartheta(\tau)$.
Then there exists a constant $d>0$ such that  $\theta _0\geq \frac{d}{\d^2}$.
\end{enumerate}
\end{proposition}
The arguments to check the existence of $Q_0$ and to get the estimation of the distance complete the results in~\cite{DIKS13},
where the required quantitative estimates were not proven.
The proof of Proposition~\ref{prop:2crossings} is deferred to Section~\ref{subsec:proofprop}.

\subsection{Proof of the existence of heteroclinic and homoclinic connections}
In order to prove that System~\eqref{formanormalanalitica2} undergoes a Shilnikov bifurcation (first item in Theorem~\ref{Theoremdissipative})
we are going to
use a classical argument, similar to the one used in~\cite{DIKS13} (see also \cite{BroerV}) to obtain the curve $\cc^*$.
But since we are in the analytic context, we have to use the accurate information we have proven about the splitting
between the stable and unstable manifolds of $p_\pm$ and about the distance between $Q_0$ and the unstable manifold of $p_-$.

Denote by $\Gamma^*_\pm$ the curves $\Gamma^*_\rho$ given in theorem in Theorem~\ref{thm:heteroclinic2d},
with $\rho_\pm=(\pm c_1,c_2,c_3)$ corresponding to the constants:
$c_2=-2-\frac{2}{a}$,  $c_3=1$ and $c_1=2\sqrt{(\mathcal{C}_1^*)^2+(\mathcal{C}_2^*)^2}$ when $\mathcal{C}_1^* + i \mathcal{C}_2^*\neq 0$
and $c_1>0$ if $\mathcal{C}_1^* + i \mathcal{C}_2^*= 0$.

When the parameters are in these curves, by Theorem~\ref{thm:heteroclinic2d}, the two dimensional stable and unstable manifolds do not intersect.
Moreover, we know that when $(\varepsilon, \sigma)\in \Gamma^*_+$  the curve $\mathbf{C}^\uns$ is outside $\mathbf{C}^\sta$ and the contrary in
$\Gamma^*_-$, and the distance between these curves is of order $\O(c_1\d^{-2-\frac{2}{a}}\mathrm{e}^{-\frac{\pi}{2a\d}}$), see Figure~\ref{fig:Trapping}.

\begin{figure}
\subfloat{
\includegraphics[width=0.45\textwidth]{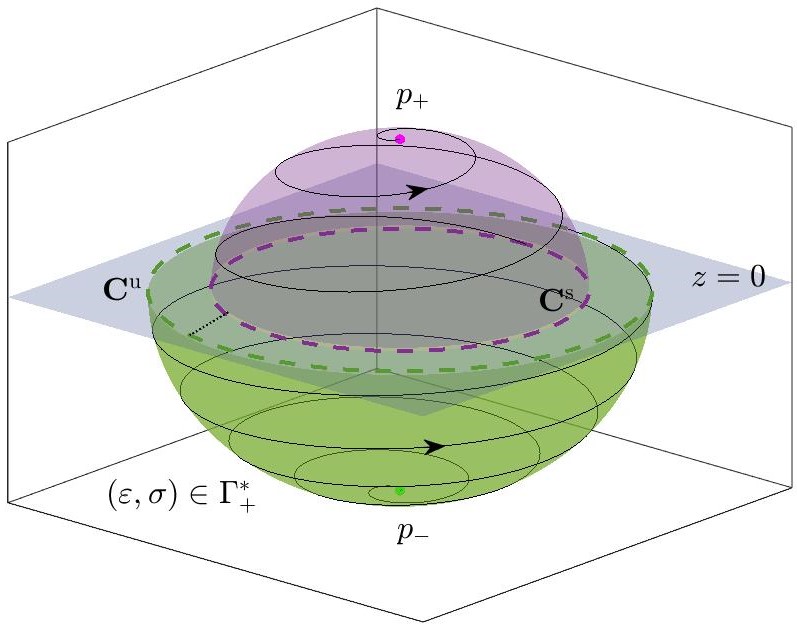}}
\hspace{0.8cm}
\subfloat{
\includegraphics[width=0.45\textwidth]{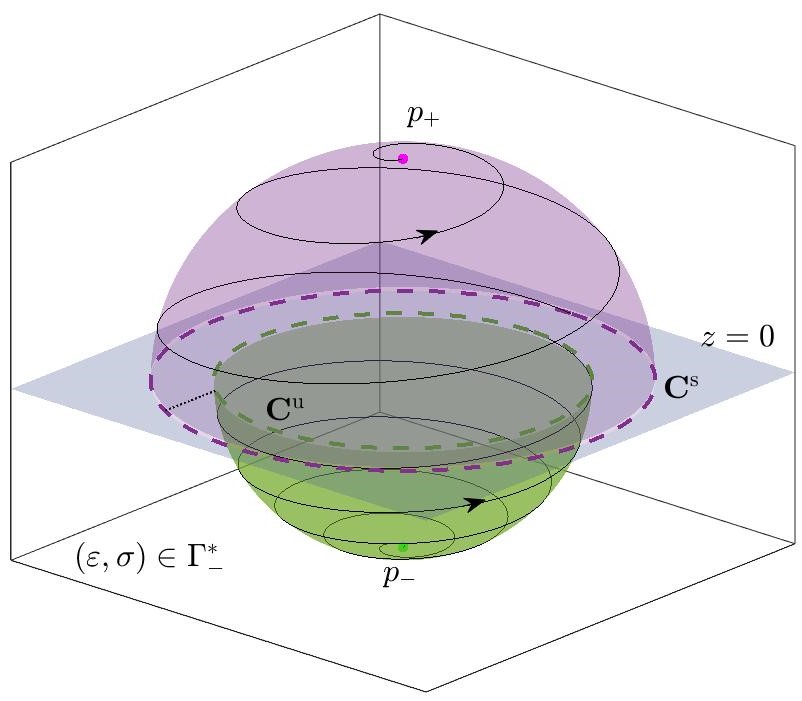}
}
\caption{On the left, the parameters $(\d,\sigma)\in \Gamma^*_+$ and $\mathbf{C}^\uns$ is outside of $\mathbf{C}^\sta$. On the right, the parameters
belongs to $\Gamma^*_-$ and $\mathbf{C}^\uns$ is inside of $\mathbf{C}^\sta$. In both case there are no intersections.}
\label{fig:Trapping}
\end{figure}

Fix $0<\varepsilon <\varepsilon_0$ and consider $\sigma \in [\sigma^*_{\rho_-}(\varepsilon), \sigma^*_{\rho_ +}(\varepsilon)]$.
By Proposition \ref{prop:2crossings} the point $Q_0$ is closer to $\mathbf{C}^\uns$ than the maximum distance between
$\mathbf{C}^\uns$ and $\mathbf{C}^\sta$.
Therefore  we can ensure  that when
$\sigma=\sigma^*_{\rho_-}(\varepsilon)$ the point $Q_0$ is inside $\mathbf{C}^\sta$, and for
$\sigma=\sigma^*_{\rho_+}(\varepsilon)$ the point $Q_0$ is outside $\mathbf{C}^\sta$, therefore, we know that it exists at least  one value
$\sigma = \sigma ^s(\varepsilon)\in [\sigma^*_{\rho_-}(\varepsilon), \sigma^*_{\rho_+}(\varepsilon)]$ where $Q_0 \in  \mathbf{C}^\sta$ and, consequently,
$W^\uns(p_+)\subset W^\sta(p_+)$ giving rise to a homoclinic orbit to $p_+$.

\begin{remark}
The previous argument resembles the proof of the existence of Shilnikov homoclinic orbits argued in \cite{DIKS13}
based in the notion of trapping region.
However, using our results we have a sharp  accurate domain on the parameters where the Shilnikov homoclinic orbits take place.
\end{remark}

This reasoning gives the existence of a homoclinic orbit to the point $p_+$ and therefore the existence of, at least, one homoclinic bifurcation for any
$0<\varepsilon<\varepsilon_0$. Then the existence of the curve $\Lambda^*$ in the parameter plane where system
\eqref{formanormalanalitica2} has Shilnikov bifurcations. Note that
$$
\Lambda^* \subset \mathcal{W}_1^* =\{ (\d,\sigma)\,:\, \sigma_{\rho_-}^*(\varepsilon ) < \sigma < \sigma_{\rho_+}^*(\varepsilon)\}
$$

With respect to the heteroclinic connections, by Theorem~\ref{thm:heteroclinic2d}, there is a heteroclinic connection
if and only if the symplectic distance $\bar{S}^2(\theta,\varepsilon,\sigma)=0$ for some $\theta$.
Let $\theta_*$ be the argument of $\mathcal{C}_1^* + i \mathcal{C}_2^*$, take any constant
$\kappa\in(0,1)$ and choose $|c_1|\leq \kappa \sqrt{(\mathcal{C}_1^*)^2+(\mathcal{C}_2^*)^2}$, $c_2=-2-\frac{2}{a}$ and $c_3=1$.
Therefore, if  $(\varepsilon,\sigma)\in \Gamma_\rho^*$, where  $\rho=(c_1,c_2,c_3)$, we need $\theta$ satisfying
\begin{equation*}
\cos\big (\theta_* - \theta +a^{-1} L_0 \log \varepsilon\big ) =
-\frac{c_1}{\sqrt{(\mathcal{C}_1^*)^2+(\mathcal{C}_2^*)^2}}(1+\mathcal{O}(\varepsilon)) + \mathcal{O}(|\log \varepsilon|^{-1}),
\end{equation*}
and this equation has two solutions if $\varepsilon$ is small enough provided $|c_1|\leq \kappa \sqrt{(\mathcal{C}_1^*)^2+(\mathcal{C}_2^*)^2}$
with $0< \kappa < 1$.

To prove the rest of the items in Theorem~\ref{Theoremdissipative} and Theorem~\ref{Theoremconservative} we apply the following theorem,
whose proof is in \cite{DIKS13}.
The theorem assumes accurate quantitative information about the relative position of the stable and unstable manifolds of the points $p_\pm$
and gives the existence of an homoclinic orbit to $p_+$.

\begin{theorem}[\cite{DIKS13}]\label{thm:diks}
Consider the  one-parameter unfolding of the Hopf-Zero singularity $X^*$ given in \eqref{formanormalanalitica2},
with $\sigma= \sigma(\varepsilon)$ given as a $\mathcal{C}^\infty$ function such that $\sigma(0)=0$.
Assume that
\begin{enumerate}
\item[H1]
There exist real constants $C_1>0$, $N_1$ and $\gamma_1>0$ such that the distance $S^1(\varepsilon,\sigma(\varepsilon))$
between the $1$-dimensional invariant manifolds satisfies:
 $$
 0<S^1(\varepsilon,\sigma(\varepsilon))\ \le C_1 \varepsilon^{N_1} \mathrm{e}^{-\frac{\gamma_1}{\varepsilon}}
 $$
\item [H2]
The radial  distance between the $2$-dimensional invariant manifolds denoted by $S^2(\theta, \varepsilon,\sigma(\varepsilon))$, can be written as:
\begin{equation}\label{eq:s2}
S^2(\theta, \varepsilon,\sigma(\varepsilon))=S^2_b(\theta,\varepsilon,\sigma(\varepsilon))+S^2_f(\varepsilon,\sigma(\varepsilon))
\end{equation}
and calling $M (\varepsilon)= \max _{\theta\in [0,2\pi]}|S^2_b(\theta,\varepsilon,\sigma(\varepsilon))|$,
there exist real constants $C_2>0$, $N_2$ and $\gamma_2>0$ such that:
$$
C_2 \varepsilon^{N_2} \mathrm{e}^{-\frac{\gamma_2}{\varepsilon}}\le M (\varepsilon)
$$
\item [H3]
The function $S^2_f(\varepsilon,\sigma(\varepsilon))$ satisfies:
$$
0\leq S^2_f(\varepsilon,\sigma(\varepsilon))\leq \tilde {C}_2 \varepsilon^{N_2} \mathrm{e}^{-\frac{\gamma_2}{\varepsilon}}
$$
for some $0<\tilde{C}_2<C_2$.
\item [H4]
The constants $\gamma_1,\gamma_2$ and $a$ satisfy:
$$
\frac{\gamma_2}{\gamma_1}<\frac{2}{a}.
$$
\item [H5]
Let $Q_0$ be the second intersection point of the unstable manifold of $p_+$ with the plane $z=0$ and write $Q_0=(r_0 \cos \theta_0,r_0 \sin \theta_0,0)$, where the angle $\theta_0=\theta_0(\varepsilon,\sigma(\varepsilon))$ is a $\mathcal{C}^\infty$ function.
There exists a $\mathcal{C}^\infty$ function $\theta_0^*(\varepsilon)$ such that
$S^2_b(\theta_0^*(\varepsilon),\varepsilon,\sigma(\varepsilon))=0$ and
$$
\theta_0(\varepsilon,\sigma(\varepsilon))-\theta_0^*(\varepsilon)\to \infty \  \text{when} \ \varepsilon \to 0.
$$
\end{enumerate}
Then, for $0<a<2$ there exists a sequence of parameter values $\varepsilon_n \to 0$, such that System~\eqref{formanormalanalitica2}
at $(\varepsilon_n, \sigma(\varepsilon_n))$ has a Shilnikov homoclinic orbit to $p^+$ which intersects the plane $z=0$ only at two points.
\end{theorem}

In what follows we check that the hypotheses of Theorem~\ref{thm:diks} are satisfied by System~\eqref{formanormalanalitica2}.
This will end the proofs of Theorems~\ref{Theoremdissipative} and~\ref{Theoremconservative}.

\begin{enumerate}
\item [H1]
By Theorem \ref{thm:splitting1d}, the first hypothesis is true taking $C_1=\mathrm{e}^{\frac{\pi c_0}{2}}\mathcal{C}^*$, $N_1=-1+a$
and $\gamma_1=\frac{\pi}{2}$.
\item [H2]
From Theorem \ref{thm:heteroclinic2d} we obtain that the distance ${S}^{2}(\theta,\varepsilon,\sigma)$ is given by~\eqref{eq:s2} with:
$$
\begin{array}{rcl}
S^2_f( \varepsilon, \sigma)
&=& \sqrt{\frac{b}{a+1}}\Upsilon^{[0]}(1+\O(\varepsilon))\\
S^2_b(\theta,\varepsilon, \sigma) &=&\sqrt{\frac{b}{a+1}}\varepsilon^{-2-\frac{2}{a}} \mathrm{e}^{-\frac{\pi}{2a\varepsilon}}
\left [\mathcal{C}_1^* \cos (\theta - a^{-1} L_0 \log \varepsilon) \right.
\\
&& \left. +
\mathcal{C}_2^* \sin (\theta - a^{-1} L_0 \log \varepsilon) + \mathcal{O}(|\log \varepsilon|^{-1})
\right ].
\end{array}
$$
Therefore hypothesis H2 is true taking
$C_2=\sqrt{\frac{b}{a+1}} \, \sqrt{(\mathcal{C}^*_1)^2 +(\mathcal{C}^*_2)^2}$, $N_2=-2-\frac{2}{a}$ and $\gamma_2=\frac{\pi}{2a}$.
\item[H3]
In the volume preserving case, because $\Upsilon^{[0]}=0$, one has $S^2_f(\sigma, \varepsilon)=0$ and
H3 is true taking any value for the constant $\tilde {C}_2\le C_2$.
The same happens in the general case, if the parameters $(\varepsilon, \sigma)$ belong to the curve
$\Gamma_0^*=\{(\varepsilon, \sigma), \  \sigma= \sigma_0^*(\varepsilon)\} $ where one has that
$\Upsilon^{[0]}(\varepsilon, \sigma_0^*(\varepsilon))=0$.

In the general case we have in addition more curves where hypothesis H3 becomes true. Indeed,
let $\kappa\in (0,1)$.
One can take the parameters along any curve $\Gamma_\rho^*=\{ (\varepsilon, \sigma),\ \sigma=\sigma^*_\rho(\varepsilon)\}$,
taking  $\rho=(c_1,c_2,c_3)$ with
$|c_1|\leq \kappa \sqrt{(\mathcal{C}^*_1)^2 +(\mathcal{C}^*_2)^2}$, $c_2\ge -2-\frac{2}{a}$, $c_3\ge 1$.
As $\Upsilon^{[0]}=c_1 \d^{c_2}  \mathrm{e}^{-\frac{c_3 \pi}{2a\varepsilon}}$ along this curve,  we have hypothesis H3 satisfied
with $\tilde{C}_2=\tilde{\kappa} \sqrt{\frac{b}{a+1}} \sqrt{(\mathcal{C}^*_1)^2 +(\mathcal{C}^*_2)^2}$,
for some $0<\kappa \leq  \tilde{\kappa}<1$.
\item[H4]
This condition is immediate:
$$
\frac{\gamma_2}{\gamma_1}= \frac{1}{a}<\frac{2}{a}
$$


\item [H5]
To check the last hypothesis, first we observe that $\theta=\theta _0^* (\varepsilon)$ has to satisfy:
\[
\cos\big (\theta_* - \theta +a^{-1} L_0 \log \varepsilon \big ) =\mathcal{O}(|\log \varepsilon|^{-1})
\]
where $\theta ^*$ is the argument of $\mathcal{C}^*_1+i\mathcal{C}^*_2$. Therefore, using the results in
Proposition~\ref{prop:2crossings} to compute $\theta_0=\theta_0(\varepsilon,\sigma (\varepsilon))$:
$$
\theta_0(\varepsilon,\sigma(\varepsilon))-\theta_0^*(\varepsilon)>\frac{d}{\d^2}-\frac{1}{a}L_0 \log \d+\frac{\pi}{2}-\theta_* +\mathcal{O}(|\log \varepsilon|^{-1}) \to \infty \
\text{when} \ \varepsilon \to 0.
$$
\end{enumerate}
Therefore, taking the parameters $(\d, \sigma)$ along any of these curves $\Gamma^* _\rho$ one can apply Theorem \ref{thm:diks} which gives
the existence of a sequence $\d_n$ such that System~\eqref{formanormalanalitica2} has an homoclinic orbit to $p_+$.
The condition $0<a<2$ ensures that the homoclinic orbit is of Shilnikov type.

\subsection{Proof of Proposition~\ref{prop:2crossings}}\label{subsec:proofprop}

As Theorem \ref{thm:splitting1d} proves that  the $1$-dimensional unstable manifold of the north pole, $W^\uns(p^-)$, intersects $z=0$ very close to
the $1$-dimensional stable manifold of the south pole $W^\sta(p^+)$,
the first step is to study the solutions of system~\eqref{formanormalanalitica2} near $W^{\sta}(p_-)$ in $z\le 0$.
More concretely, we want to see that orbits entering the plane $z=0$ near
the $1$-dimensional stable manifold
$W^\sta(p_-)$, leave the plane again near the $2$-dimensional unstable manifold $W^\uns(p_-)$.

In order to control de behavior of the solutions which enter $z\le 0$ near the $1$-dimensional stable manifold $W^\sta(p_-)$, in particular $W^\uns(p_+)$,
we perform an analytic change of coordinates such that the point $p_-$ becomes $(0,0,-1)$ and
its stable manifold $W^{\sta}(p_-)$ becomes the $z$-axis.
As this change does not affect the terms in normal form, we obtain the following system, that we write keeping the notation $(x,y,z)$:
\begin{equation}\label{sistemaejez}
 \begin{array}{rcl}
  \displaystyle\frac{dx}{dt}&=&\displaystyle x\left(\s-\dd z\right)+\frac{y}{\d}+
  \d \tilde{f} (x,y,z, \d,\s),\medskip\\
  \displaystyle\frac{dy}{dt}&=&\displaystyle-\frac{x}{\d}+y\left(\s-\dd z\right)+
  \d \tilde{g} (x,y,z, \d,\s),\medskip\\
  \displaystyle\frac{dz}{dt}&=&-1+b(x^2+y^2)+z^2+\d \tilde{h} (x,y,z, \d,\s),
 \end{array}
\end{equation}
where $\tilde f(x,y,z,\d,\s), \tilde g(x,y,z,\d,\s)=\O(\|(x,y)\|)$ and the function $\tilde h$ satisfies $\tilde h(x,y,z,\d,\s)=\O(\|(x,y,z+1)\|)$.
To finish we perform a suitable linear change of variables, $\mathcal{O}(\d^2)$ close to the identity, which is explained with detail in~\cite[Lemma 4.1]{BCS16a}
to put $\tilde h$ of the form
\begin{equation*}
\tilde{h} (x,y,z,\d,\s)=(z+1)\tilde{h}_1 (x,y,z,\d,\s) + \tilde{h}_2(x,y,z,\d,\s),
\end{equation*}
with $\tilde{h}_2(x,y,z,\d,\s)=\mathcal{O}(\|x,y,z+1\|^2)$ and $\tilde{h}_1$ bounded.
In addition, the corresponding $\tilde{f}$, $\tilde{g}$ keep the same properties.

From the results in~\cite{BF05} and~\cite[Proposition 4.4]{BCS16a}, it can be easily deduced that
\begin{equation}\label{aproxvariedad2D}
W^{\uns}(p_-)\cap \{(x,y,z),\ z\le 0\}=\left \{\frac{b}{a+1}(x^2+ y^2) + z^2 = 1+  \d  (1-z^2) \psi(z,\theta)\right \},
\end{equation}
where $\psi(z,\theta)$ is an analytic  periodic function in $\theta$ satisfying that there exists a constant $M>0$ such that for $z\le 0$, $\theta \in [0,2\pi)$:
\begin{equation*}
|\psi(\theta, z)|\le M, \qquad
\left |\frac{\partial \psi}{\partial z}(\theta, z)\right|\le M, \qquad
\left |\frac{\partial \psi}{\partial \theta}(\theta, z)\right |\le \varepsilon M.
\end{equation*}
That is, the $2$-dimensional unstable manifold of the point $p_-=(0,0,-1)$ in the domain $z\le 0$ is $\varepsilon$-close to the ellipsoid \eqref{eq:sfere}.

\begin{lemma}\label{lemadif}
Let $D=\overline{D}$ be the closed region $D\subset \{(x,y,z)\in \mathbb{R}^3 : z\leq 0\}$ with boundaries:
$$
\partial D = \{z=0\} \cup W^{\uns}(p_-).
$$
Take $\zeta_1(t), \zeta_2(t)$ two solutions of system~\eqref{sistemaejez}  such that
$$
\zeta_1(0),\zeta_2(0)\in \partial D_0:=\{z=0\} \cap \left \{b(x^2+y^2)\leq \frac{1}{2}\right \}.
$$
Then, there exist $\tau,C>0$ such that $\zeta_1(t),\zeta_2(t) \in D$ for all $t\in [0,\tau]$
and
$$
\|\zeta_1(t) -\zeta_2(t)\| \leq \|\zeta_1(0)-\zeta_2(0)\| \mathrm{e}^{C\,t },\qquad t\in [0,\tau].
$$
\end{lemma}
\begin{proof}
First we note that $\partial D_0 \subset \partial D\cap \{z=0\}$, because the maximum radius in $\partial D$ is $\sqrt{\frac{\dd +1}{b}} + \O(\d)$
which is greater than $\sqrt{\frac{1}{2b}}$ the maximum radius in $\partial D_0$.
The fact that $\tau>0$ is a direct consequence from the fact that, on $\partial D_0$, if $\varepsilon$ is small enough,
we have that $\dot{z}<-\frac{1}{2}+ \O(\d)<0$ and therefore the solutions $\zeta_1(t)$ and $\zeta_2(t)$ enter in $\text{int}{D}$.

We consider now $\Delta \zeta=\zeta_1-\zeta_2$.
As both are solutions of System~\eqref{sistemaejez} which is an analytical vector field, and we know that
$\zeta_1(t),\zeta_2(t)\in D$ for $0\le t\le \tau$ and $D$ is a bounded region, it is clear, by the mean value Theorem, that
$\Delta \zeta$ is a solution of the homogeneous linear equation
$$
\frac{d \Delta \zeta}{dt} = A \Delta \zeta + \mathcal{A}(t) \Delta \zeta, \qquad \max_{t\in [0,\tau]} \| \mathcal{A}(t)\| \leq K,
$$
being $A$ the matrix
\begin{equation*}
A =\left (\begin{array}{ccc} \s+\dd  & \frac{1}{\d} & 0 \\
-\frac{1}{\d} & \s+\dd & 0 \\ 0 & 0 & 0
\end{array}\right ).
\end{equation*}

Then, since $\| \mathrm{e}^{At}\| \leq  \mathrm{e}^{(a+\s)t}$ where $\| \cdot \|$ is the euclidian norm, we have that
$$
\|\Delta \zeta (t) \mathrm{e}^{-(a+\s) t}\|\leq  \|\Delta \zeta(0)\| +  K \int_{0}^t   \mathrm{e}^{-(a+\s)s} \|\Delta \zeta(s)\|\, ds
$$
and by Gronwall's lemma one deduces
$$
\|\Delta \zeta (t) \mathrm{e}^{-(\s+a)t}\|\leq  \|\Delta \zeta(0)\| \mathrm{e}^{K t }
$$
which gives the result taking $C=K+\sigma+a$.
\end{proof}

Next two lemmas are devoted to show how the flow of~\eqref{sistemaejez} behaves on suitable surfaces.

\begin{lemma}\label{cylinder}
For any $R_0>0$, there exists $\d_0>0$ small enough such that for any $\d\in (0,\d_0)$, $|\sigma|\le  \sigma_0 \d$ (see Theorem \ref{thm:heteroclinic2d})
and $0<R\leq R_0$ if one considers the cylinder
$$
C_R=\{ (x,y,z), \ x^2 +y^2 =R^2\}\cap \{(x,y,z), \ -1\leq z\leq -\d|\log \d| \}.
$$
the flow of system~\eqref{sistemaejez} is pointing outwards  the side boundary of $C_R$.
\end{lemma}

\begin{proof}
The normal vector to $\partial C_R$
is $(x,y,0)$. Therefore we need to check that $x\dot{x}+y\dot{y}>0$ for $(x,y,z) \in \partial C_R$:
$$
x\dot{x}+y\dot{y}= (\s - \dd z)R^2 + \d \O(R^2) \geq  \dd \d R^2 (|\log \d| +\O(1)))>0
$$
if $\d$ is small enough.
\end{proof}

\begin{lemma}\label{elipsoide}
Take $\nu_1>0$ and the ellipsoid
$$
S_{\nu_1}=\left \{z^2 + \frac{b}{\dd+1}(x^2+y^2)=1-\nu_1\right \}.
$$
Then there exists $\d_0>0$ small enough such that for any $\d\in (0,\d_0)$, $|\sigma|\le  \sigma_0 \d$ the surface defined by
$S_{\nu_1}\cap\{-1\leq z \leq -\d |\log \d|\} $ is contained in the region $D$ defined in Lemma \ref{lemadif} and the flow  of system~\eqref{sistemaejez} points outwards.
\end{lemma}
\begin{proof}
As the unstable manifold of $p_-$, $W^\uns(p_-)$, is $\varepsilon$-close to the ellipsoid \eqref{eq:sfere} (see~\eqref{aproxvariedad2D})
in the region $z\le 0$ and $\nu_1$ is small but independent of
$\varepsilon$, is clear that if
$\varepsilon$ is small enough $S_{\nu_1}\cap\{-1\leq z \leq -\d |\log \d|\} $ is contained in the region $D$. Moreover:
\begin{align*}
z\dot{z} + \frac{b}{\dd +1} (x\dot{x} + y\dot{y}) =& z(-1 + z^2 + b(x^2+ y^2) + \d\O(\|x,y,z+1\|))
\\ & + \frac{b}{\dd +1} (x^2+ y^2) (\s - \dd z + \O(\d))
\\ =& z(\dd - \nu_1(\dd+1)-\dd z^2 +\d \O(\|x,y,z+1\|))\\
&+ (1-\nu_1-z^2)(- \dd z + \O(\d))\\
=& -\nu_1 z +\O(\d) \geq \d( \nu_1 |\log \d| + \O(1))>0
\end{align*}
if $\d$ is small enough.
\end{proof}

Recall that in  the coordinates which give System~\eqref{sistemaejez} the ``north'' equilibrium point is of the form $p_+=(0,0,1)+\O(\d)$.

\begin{proof}[End of the proof of Proposition~\ref{prop:2crossings}]
By Theorem \ref{thm:splitting1d} we know that  the points
$$
\zeta^{\uns}=W^{\uns}(p_+)\cap \{z=0\},\qquad \zeta^{\sta}=W^{\sta}(p_-) \cap \{z=0\},
$$
satisfy that, for some constant $\bar {\mathcal{C}}^*$,
\begin{equation}\label{expdem}
\|\zeta^{\uns}-\zeta^{\sta}\| = \d^{(-1+\dd)} \mathrm{e}^{-\frac{ \pi}{2 \d}}(\bar {\mathcal{C}}^*+\O(\d)).
\end{equation}
Let us consider $\zeta_{\uns}(t)$ and $\zeta_{\sta}(t)$ the solutions of system~\eqref{sistemaejez} such that
$\zeta_{\uns,\sta}(0)=\zeta^{\uns,\sta}$.
It is clear that, $\zeta_{\sta}$ is defined for all  $t\geq 0$ and  $\lim_{t\to +\infty} \zeta_{\sta}(t)= p_-=(0,0,-1)$,
therefore it belongs to the domain $D$ defined in Lemma~\ref{lemadif}.
In fact $\zeta_{\sta}(t) \in \{x=y=0\}$ for all $t\geq 0$.
Let $\tau>0$, given by this lemma, be such that $\zeta_{\uns}(t)\in D$ for $t\in[0,\tau]$ .
Then, applying Lemma~\ref{lemadif}, there exist a positive constant $C$, such that, for $0\le t \le \tau$:
\begin{equation}\label{desstauns}
\|\zeta_{\uns}(t)-\zeta_{\sta}(t)\| \leq  \frac{\bar {\mathcal{C}}^*}{\d^{1-\dd}} \mathrm{e}^{-\frac{ \pi}{2 \d}} \mathrm{e}^{C\, t}.
\end{equation}
As $\zeta_{\sta}(t) \to p_-=(0,0,-1)$ as
$t \to +\infty$,
one concludes that, in order to leave the domain $D$, $\tau$ has to satisfy
$$
\frac{1}{\d^{1-\dd}} \mathrm{e}^{-\frac{\pi}{2 \d}} \mathrm{e}^{C \tau} = \O(1)
$$
and therefore, we can assure that, at least
\begin{equation}\label{eq:tau}
\tau \geq \frac{\pi}{4C\d}.
\end{equation}
Observe that the lower bound for $\tau$ is not sharp but it will be enough for our purposes.

Now we prove that indeed, $\zeta_{\uns}(t)$ leaves the domain $D$, that is, it crosses the plane $\{z=0\}$ for some $t\geq \tau$.
With this result we prove the second item in Proposition~\ref{prop:2crossings}.
We proceed by assuming the contrary, that is, that $\zeta_{\uns}(t) \in \text{int} D $ for all $t>0$.
We do the proof in three steps:
\begin{itemize}
\item [\textbf{Step 1}]. Preliminary considerations.

Take $\nu_1>0$ and consider the closed region $V_{\nu_1}$ with boundary
$$
\partial V_{\nu_1}= W^{\uns}(p_-) \cup \{z=-\d |\log \d| \} \cup S_{\nu_1}
$$
where the ellipsoid $S_{\nu_1}$ is defined in Lemma \ref{elipsoide}.
Recall that, as it is indicated in~\eqref{aproxvariedad2D}, up to order $\d$ the unstable manifold $W^{\uns}(p_-)$ is well approximated by
the ellipsoid
$$
z^2 + \frac{b}{\dd+1}(x^2+y^2)=1.
$$
Consequently, the region $V_{\nu_1}$ is $\O(\d)$ close to the region
\begin{align*}
\{z\leq-\d & |\log \d|\} \bigcap \\
&\left \{-\sqrt{ 1-\frac{b}{\dd +1}(x^2+y^2)} \leq z\leq -\sqrt{ 1-\nu_1-\frac{b}{\dd +1}(x^2+y^2)}\right \}.
\end{align*}
\item [\textbf{Step 2}]. Leaving $V_{\nu_1}$ by $\{z=-\d |\log \d |\}$.

We claim that there exists $\tau_0>0$ independent of $\d$ such that
$$
\zeta_{\uns}(\tau_0) \in S_{\nu_1}.
$$
Indeed, in~\cite{BCS13} was proven that $\zeta_{\sta}(t) = (0,0,\tanh(-t)+\O(\d))$.
Then we can go down along $\zeta_{\uns}(t)$ as close to $z=-1$ as we want. Since $\nu_1$ is independent of $\d$ we can ensure that there exists  $\bar \tau _0$
independent of $\varepsilon$ such that $\pi^z \zeta_{\sta}(\bar \tau _0) =-1 +\nu_1/2$, consequently, as~\eqref{desstauns} assures
that $\zeta_{\uns}(t)$ will be exponentially close to $\zeta_{\sta}(t)$, there exists $\tau_0$ independent of $\varepsilon$ such that
$\zeta_{\uns}(\tau _0)\in S_{\nu_1}$ and enters in $V_{\nu_1}$.

Let $\tau_1$ be such that $\zeta_{\uns}(t) \in V_{\nu_1}$ if $t\in [\tau_0,\tau_1]$.
By Lemma~\ref{elipsoide}, the only way to leave the region $V_{\nu_1}$ is to cross the section $\{z=-\d|\log \d|\}$.
In addition, Lemma~\ref{cylinder} assures that $\zeta_{\uns}(t)$ leaves every cylinder $\{x^2 + y^2 =R^2\}$. Therefore,
we conclude that $\zeta_{\uns}(\tau_1) \in \{z=-\d |\log \d|\}$.

\item [\textbf{Step 3}]. Final conclusion.

Since we are assuming that $\zeta_{\uns}(t) \in \text{int} D$ for all $t>0$,
we have now that if $t\geq \tau_1$, then $\pi^z \zeta_{\uns}(t) \in [-\d |\log \d|,0)$. We write
$\zeta_{\uns}(t)=(x_{\uns}(t) , y_{\uns}(t),z_{\uns}(t))$.
It is now clear that $b (x_{\uns}^2 (\tau_1) + y_{\uns}^2 (\tau_1))\geq (\dd +1)(1-\nu_1 + \O(\d^2|\log \d|^2))$.
As, by Taylor's formula,  there exists $\xi_{t}\in [\tau_1,t]$ such that
\begin{align*}
z_{\uns}(t) = &z_{\uns}(\tau_1)+ \dot{z}_{\uns}(\tau_1) (t-\tau_1) + \frac{1}{2}\ddot{z}_{\uns}(\xi_t) (t-\tau_1)^2
\\= & -\d |\log \d| + (-1 + \d^2 \log^2 \d + b(x_{\uns}^2 (\tau_1) + y_{\uns}^2 (\tau_1))+ \O(\d)) (t-\tau_1)
\\ & + \frac{1}{2}\ddot{z}_{\uns}(\xi_t) (t-\tau_1)^2
\\ \geq & -\d |\log \d|+ (-1 + (\dd +1)(1-\nu_1) + \O(\d))(t-\tau_1) \\ &+  \frac{1}{2}\ddot{z}_{\uns}(\xi_t) (t-\tau_1)^2.
\end{align*}
Take now $t_m=\tau_1 + m\d |\log \d|$ with $m>0$. Since the solution $\zeta_{\uns}(s)$ remains bounded
for $s\geq \tau_1$,
$$
z_{\uns}(t_m) \geq \d |\log \d|\big (-1 -m+ m (\dd +1)(1-\nu_1) \big )+ \O(\d^2 |\log \d|^2).
$$
Take $\nu_1>0$ such that $(\dd +1 )\nu_1 <\dd$, then $(\dd +1) (1-\nu_1)-1>0$ and hence for
$$
m>\frac{2}{(\dd +1)(1-\nu_1)-1}
$$
we have that $z_{\uns}(t_m) \geq \d |\log \d|+ \O(\d^2 |\log \d|^2)>0$ which is a contradiction since we have assumed that $z_{\uns}(t)<0$ if $t>0$.
\end{itemize}

If we call $Q_0=(r_0\cos \theta_0,r_0\sin \theta_0,0)=\zeta_{\uns}(\tau)$ the point where $z_{\uns}(\tau)=0$,
and write the equations \eqref{sistemaejez}  in cylindrical coordinates we have that
$$
\frac{d\theta}{dt}=\frac{-1}{\d}+\O(\d)
$$
therefore, the estimate for $\theta_0$ in Proposition~\ref{prop:2crossings} comes from~\eqref{eq:tau}.

It remains to estimate the distance from $Q_0$ to the unstable circle $\mathbf{C}^\uns$. Let us introduce
$r^2=x^2+y^2$. It is a well known fact that, as $\sigma = \O(\d)$,
System~\eqref{sistemaejez} written in cylindrical coordinates is $\O(\d)$-close to an integrable system with first integral
$$
H(r,z)=r^{2/a}\left[ 1-z^2-\frac{b}{a+1}r^2\right].
$$
We introduce the new variable
\begin{equation}\label{eq:rho}
\rho = r(1+ \d \psi(z,\theta))^{-1/2},
\end{equation}
with $\psi$ defined by~\eqref{aproxvariedad2D}.
Then the two dimensional unstable manifold
of $p_-=(0,0,-1)$ is
$$
W^{\uns}(p_-)=\left \{ z^2 + \frac{b}{a+1} \rho^2 =1\right \}.
$$
The variation with respect to $t$ of $(\rho,z)$
is
\begin{equation}\label{edorho}
\begin{aligned}
\frac{d \rho}{dt} &= ( \s - a z)\rho + \d \rho \widehat{f}(\rho, z,\theta, \d,\s) \\
\frac{d z }{dt} &= -1+z^2 + b\rho^2 (1+\d \psi)+ \d (z+1) \widehat{h}_1 (\rho, z,\theta, \d,\s) + \d \widehat{h}_2 (\rho, z,\theta, \d,\s)
\end{aligned}
\end{equation}
with $\widehat{f},\widehat{h}_1$ bounded and $\widehat{h}_2$ of $\mathcal{O}(\|\rho, z+1\|^2)$.
That is, the properties of $\tilde{f}$ and $\tilde{h}_1$, $\tilde{h}_2$
are satisfied also after the change of variables.

\begin{itemize}
\item
[\textbf{Step 1}].
We first prove that, if $(\rho(t),z(t))\in \{H(\rho,z)\geq 0\} \cap \{z\leq 0\}$ for
$t\geq t_0$, then
\begin{equation}\label{casiH}
H(\rho(t),z(t)) \leq H(\rho(t_0),z(t_0)) \mathrm{e}^{K\d (t-t_0)}
\end{equation}
for some constant $K$ independent of the initial condition.

It is not difficult to see that
\begin{align*}
\frac{dH}{dt}(\rho,z) = & \frac{2}{a} H(\rho,z) \big (\s + \d \widehat{f}(\rho, z,\theta, \d,\s)\big ) \\
&-2z \rho^{2/a} \big (\d(z+1)  \widehat{h}_1 (\rho, z,\theta, \d,\s) +
\d \widehat{h}_2 (\rho, z,\theta, \d,\s) \big ) \\
&-\rho^{2+\frac{2}{\dd}} \frac{2b}{\dd +1} \big (\sigma  + \d \widehat{f}(\rho, z,\theta, \d,\s)\big ).
\end{align*}
Let us denote $\overline{\sigma}=\sigma\d^{-1} = \mathcal{O}(1)$ and
$$
\overline{h}_2(\rho,z,\theta,\d,\s) = 2z \widehat{h}_2 (\rho, z,\theta, \d,\s)
+\rho^2 \frac{2b}{\dd +1}\big (\overline{\sigma} +  \widehat{f}(\rho, z,\theta, \d,\s)\big ).
$$
The unstable manifold of $p_-$ is contained in $H(\rho,z)=0$ and therefore is given by
$\rho = \varphi(z) := \sqrt\frac{a+1}{b} (1-z^2)^{1/2}$. Consequently
$$
\frac{dH}{dt}(\varphi(z),z)\equiv 0.
$$
and using that $H(\varphi(z),z)\equiv 0 $, we obtain
$$
(z+1)z  \widehat{h}_1 (\varphi(z), z,\theta, \d,\s) + \overline{h}_2 (\varphi(z), z,\theta, \d,\s)\equiv 0.
$$
Notice that $D _{\rho} \widehat{h}_1$ is bounded and
$D _{\rho} \overline{h}_2 =\mathcal{O}(\|(\rho, z+1)\|)$.
Then, using the mean's value theorem we have that
\begin{align*}
\big |\widehat{h}_1 (\rho, z,\theta, \d,\s) &-\widehat{h}_1 (\varphi(z), z,\theta, \d,\s) \big | \\
 =&\left |\int_{0}^1 D _{\rho}\widehat{h}_1 (\rho + \lambda( \varphi(z) - \rho), z,\theta, \d,\s) \, d\lambda \big (\varphi(z) - \rho)
\right | \\
\leq &K |\varphi(z)-\rho| \\
\big |\widehat{h}_2 (\rho, z,\theta, \d,\s)&-\widehat{h}_2 (\varphi(z), z,\theta, \d,\s) \big | \\
=&\left |\int_{0}^1 D _{\rho} \overline{h}_2 (\rho + \lambda( \varphi(z) - \rho), z,\theta, \d,\s) \, d\lambda \big (\varphi(z) - \rho)
\right | \\
\leq &K  (|z+1|+\rho + |\varphi(z)-\rho|) |\varphi(z)-\rho|.
\end{align*}
As a consequence
\begin{equation}\label{bounddHt}
\left |\frac{dH}{dt}(\rho,z)  \right |\leq \d K H (\rho,z) + K \d\rho^{2/a}
\big (|z+1| + \rho + |\varphi(z)-\rho|\big ) |\varphi(z)-\rho|.
\end{equation}
Recall that $\varphi(z)=\sqrt{\frac{a+1}{b}} \sqrt{1-z^2}$. Notice on the one hand that
\begin{align*}
H(\rho,z) & = \rho^{2/a} \frac{b}{a+1} \left (\frac{a+1}{b} (1-z^2) -\rho^2 \right )\\
&=\rho^{2/a}\frac{b}{a+1} \big (\varphi(z)+ \rho \big )(\varphi(z) -\rho)
\end{align*}
which implies
$$
\rho^{2/a} |\varphi(z)- \rho|= \sqrt{\frac{\dd+1}{b}} H(\rho,z) \left (\sqrt{1-z^2} + \rho \sqrt{\frac{b}{a+1}} \right )^{-1}
$$
On the other hand, using that $\varphi(z),\rho>0$, we have that
$|\varphi(z)-\rho| \leq |\varphi(z)+\rho |$.
Therefore, using that $z\leq 0$, from~\eqref{bounddHt} and changing $K$ if necessary, we obtain:
\begin{align*}
\left |\frac{dH}{dt}(\rho,z)  \right |&\leq \d K H (\rho,z) \left (1 + \frac{|z+1| + \rho + |\varphi(z)-\rho|}
{\sqrt{1-z^2} + \sqrt{\frac{b}{a+1} \rho}} \right )\\
& \leq \d K H(\rho,z) \left (1 + \sqrt{\frac{|z+1|}{|1-z|}} + \frac{b}{a+1} + 1\right )\\
&\leq \d K H(\rho,z)
\end{align*}
This inequality implies, by Gronwall's lemma, bound~\eqref{casiH}, provided
$z(t_0)\leq 0$.
\item
[\textbf{Step 2}].
Now we provide a bound from below of $\tau$, the time that the solution $\zeta_{\uns}(t)$ needs for crossing again $\{z=0\}$.
Recall that $\zeta_{\uns,\sta}(t)$ are the solutions such that $\zeta_{\uns,\sta}(0)=\zeta^{\uns,\sta}\in \{z=0\}$.
We call $(r_\uns, \theta_\uns,z_\uns)$  the corresponding cylindrical coordinates and $\rho_{\uns}(t)$ the corresponding radius defined in \eqref{eq:rho}.
We will prove that,
\begin{equation}\label{boundt0}
\rho_\uns(\tau)=\mathcal{O}(1),\qquad \tau\leq \frac{d_0}{\d}.
\end{equation}
for some constant $d_0$.

First we consider  $t_1>0$ such that
$\zeta_{\uns}(t)$ crosses by the first time the plane $z=-\eta$ for some $\eta>0$.
We have that $z_{\uns}(t_1)=-\eta$. As argued before, $t_1$ is independent of $\d$ and the corresponding radius
$\rho^\uns (t_1)\sim K \varepsilon^{-1+a} \mathrm{e}^{-\frac{\pi}{2\varepsilon}}$, that is,
it  is of the same order as $\|\zeta^\uns - \zeta^\sta\|$ which is written in~\eqref{expdem}.

Now we consider the minimum value $t_2>t_1$  such that $\zeta_{\uns}(t_2) \in \{z=-\eta\}$.
We have that for $t\in [t_1,t_2]$, $z_{\uns}(t) \in [-1 , -\eta]$.
Then, by \eqref{edorho} we have that
$$
\dot{\rho}_{\uns} =-a z_{\uns} \rho_{\uns} + \rho_{\uns} \mathcal{O}(\d) \geq c_1 \rho_\uns
$$
with $c_1=a\eta + \mathcal{O}(\d)$.
This implies that
\begin{equation}\label{lowerboundr}
\rho_{\uns}(t) \geq \rho_{\uns}(t_1) \mathrm{e}^{c_1 (t-t_1)}, \qquad t\in [t_1,t_2].
\end{equation}
Assume $\eta$ small enough to take $\nu_1 < 1-\eta^2$, with $\nu_1>0$.
Note that
$$
z_\uns^2(t_1) + \frac{b}{a+1} \rho_\uns^2 (t_1) =\eta^2 + \mathcal{O}\big(\d^{-1+\dd} \mathrm{e}^{-\frac{ \pi}{2 \d}} \big ) <1-\nu_1,
$$
and hence, $(\rho_u(t_1),z_u(t_1))\notin V_{\nu_1}$.
Then, reasoning as before, there exists $t_1'>t_1$ independent of $\d$ such that
$\zeta_{\uns}(t)\in V_{\nu_1}$ for $t\in [t_1',t_2]$.
Then
$$
\frac{b}{a+1}\rho_{\uns}^2(t_2) \geq 1-\nu_1-z^2_{\uns}(t_2) =1- \nu_1 -\eta^2>0
$$
which is an independent of $\d$ constant.
Consequently, when $\zeta_{\uns}$ crosses the plane $\{z=-\eta\}$ the radius $\rho_{\uns}(t_2)=\mathcal{O}(1)$.
We recall that
$\rho_{\uns}(t_1)\sim K \varepsilon^{-1+a} \mathrm{e}^{-\frac{\pi}{2\varepsilon}}$.
Then by~\eqref{lowerboundr} $t_2$ has to satisfy
$$
t_2 -t_1
\leq  \frac{1}{c_1} \left [\log \rho_{\uns}(t_2) + \frac{\pi}{2\d} + (a-1) \log \d\right ]
\leq \frac{c_2}{\d}
$$
for some positive constant $c_2$.

The time $\tau$ $\zeta_{\uns}$ needs to meet  the plane $z=0$ satisfies $\tau-t_2 = \mathcal{O}(1)$.
Indeed, on the one hand, meanwhile the solution is in $V_{\nu_1}$, that is if $z_{\uns}(t) \leq -\d|\log \d|$,
we have that
$$
\dot{z}_{\uns} = -1 + b \rho_{\uns}^2 + z_{\uns}^2+\mathcal{O}(\d) \geq -1 + (a+1)(1- \nu_1 -\eta^2) +\mathcal{O}(\d) \geq
\frac{a}{2}
$$
if we take $\nu_1, \eta$ and $\d$ small enough.
Then, in this case we get $z=-\d|\log \d|$ in a finite time. On the other hand
the transition from $z=-\d|\log \d|$ has been studied before obtaining also a finite (in fact of $\mathcal{O}(\d|\log \d|)$) time.
\item
[\textbf{Step 3}].
Final conclusion.
Using~\eqref{casiH} with $t_0=0$ and $t=\tau$ and~\eqref{boundt0}:
$$
H(\rho_{\uns}(\tau),z_{\uns}(\tau)) \leq K H(\rho_{\uns}(0),z_{\uns}(0)).
$$
Again by~\eqref{boundt0}, we deduce from the above bound that
$$
1 - \frac{b}{a+1} \rho_{\uns}^2(\tau) \leq K\rho_{\uns}(0)^{2/a}(1+ \mathcal{O}(\rho_{\uns}^2(0)).
$$
It follows that, there exists a constant $L>0$ such that
$$
\sqrt{\frac{a+1}{b}}- \rho_{\uns}(\tau)  \leq L \rho_{\uns}(0)^{2/a}.
$$
The required bound for the distance between $Q_0$ and the unstable curve $\mathbf{C}^u$ follows from the above inequality taking into account (\ref{expdem}) and that $\mathbf{C}^u$ is the circumference of radius $\sqrt{\frac{a+1}{b}}$.
\end{itemize}
\end{proof}

\section*{Acknowledgements}
I. Baldom\'a and T.M. Seara have been partially supported by the Spanish Government
MINECO-FEDER grant MTM2015-65715-P and the Catalan Government grant 2017SGR1049. S. Ib\'{a}\~{n}ez has been supported by the
Spanish Research project MTM2014-56953-P.


\begin{thebibliography}{alpha}

\bibitem[BC91]{BC91}
M.~Benedicks and L.~Carleson.
\newblock The dynamics of the {H}\'enon map.
\newblock {\em Ann. of Math. (2)}, 133(1):73--169, 1991.

\bibitem[BCS13]{BCS13}
I.~Baldom{\'a}, O.~Castej{\'o}n, and T.~M. Seara.
\newblock {Exponentially small heteroclinic breakdown in the generic Hopf-Zero
  singularity}.
\newblock {\em Journal of Dynamics and Differential Equations}, 25(2):335--392,
  2013.

\bibitem[BCS16a]{BCS16a}
I.~Baldom{\'a}, O.~Castej{\'o}n, and T.~M. Seara.
\newblock {Breakdown of a 2D heteroclinic connection in the Hopf-zero
  singularity (I)}.
\newblock {\em preprint}, 2016.

\bibitem[BCS16b]{BCS16b}
I.~Baldom{\'a}, O.~Castej{\'o}n, and T.~M. Seara.
\newblock {Breakdown of a 2D heteroclinic connection in the Hopf-zero
  singularity (II). The generic case}.
\newblock {\em preprint}, 2016.

\bibitem[BDV05]{bondiavia}
C.~Bonatti, L.~J. D{\'{\i}}az, and M.~Viana.
\newblock {\em Dynamics beyond uniform hyperbolicity}, volume 102 of {\em
  Encyclopaedia of Mathematical Sciences}.
\newblock Springer-Verlag, Berlin, 2005.
\newblock A global geometric and probabilistic perspective, Mathematical
  Physics, III.

\bibitem[BF03]{BF03}
P.~Bonckaert and E.~Fontich.
\newblock Invariant manifolds of maps close to a product of rotations: close to
  the rotation axis.
\newblock {\em Journal of Differential Equations}, 191(2):490 -- 517, 2003.

\bibitem[BF05]{BF05}
P.~Bonckaert and E.~Fontich.
\newblock Invariant manifolds of dynamical systems close to a rotation:
  transverse to the rotation axis.
\newblock {\em J. Differential Equations}, 214(1):128--155, 2005.

\bibitem[Bir35]{Bir35}
G.~D. Birkhoff.
\newblock {Nouvelles recherches sur les systèmes dynamiques}.
\newblock {\em Memoriae Pont. Acad. Sci. Novi Lyncaei}, 1:85--216, 1935.

\bibitem[BIR11]{BIR11}
P.~G. Barrientos, S.~Ib{\'a}{\~n}ez, and J.~A. Rodr{\'\i}guez.
\newblock Heteroclinic cycles arising in generic unfoldings of nilpotent
  singularities.
\newblock {\em J. Dynam. Differential Equations}, 23(4):999--1028, 2011.

\bibitem[Bro81]{broer81}
H.~Broer.
\newblock Formal normal form theorems for vector fields and some consequences
  for bifurcations in the volume preserving case.
\newblock In {\em Dynamical systems and turbulence, {W}arwick 1980 ({C}oventry,
  1979/1980)}, volume 898 of {\em Lecture Notes in Math.}, pages 54--74.
  Springer, Berlin-New York, 1981.

\bibitem[BS06]{BaSe06}
I.~Baldom{\'a} and T.~M. Seara.
\newblock {Breakdown of heteroclinic orbits for some analytic unfoldings of the
  Hopf-zero singularity.}
\newblock {\em J. Nonlinear Sci.}, 16(6):543--582, 2006.

\bibitem[BV84]{BroerV}
H.~W. Broer and G.~Vegter.
\newblock Subordinate \v {S}ilnikov bifurcations near some singularities of
  vector fields having low codimension.
\newblock {\em Ergodic Theory Dynam. Systems}, 4(4):509--525, 1984.

\bibitem[CK04]{CK2004}
A.~R. Champneys and V.~Kirk.
\newblock The entwined wiggling of homoclinic curves emerging from
  saddle-node/{H}opf instabilities.
\newblock {\em Phys. D}, 195(1-2):77--105, 2004.

\bibitem[CY08]{CY2008}
S.~A. Campbell and Y.~Yuan.
\newblock Zero singularities of codimension two and three in delay differential
  equations.
\newblock {\em Nonlinearity}, 21(11):2671--2691, 2008.

\bibitem[DIK01]{dumibakok2001}
F.~Dumortier, S.~Ib{\'a}{\~n}ez, and H.~Kokubu.
\newblock New aspects in the unfolding of the nilpotent singularity of
  codimension three.
\newblock {\em Dyn. Syst.}, 16(1):63--95, 2001.

\bibitem[DIK06]{dumibakok2006}
F.~Dumortier, S.~Ib{\'a}{\~n}ez, and H.~Kokubu.
\newblock Cocoon bifurcation in three-dimensional reversible vector fields.
\newblock {\em Nonlinearity}, 19(2):305--328, 2006.

\bibitem[DIKS13]{DIKS13}
F.~Dumortier, S.~Ib{\'a\~n}ez, H.~Kokubu, and Carles Sim\'o.
\newblock About the unfolding of a {H}opf-zero singularity.
\newblock {\em Discrete Contin. Dyn. Syst.}, 33(10):4435--4471, 2013.

\bibitem[DIP19]{DIP2019}
F.~Drubi, S.~Ib{\'a\~n}ez, and P.~Pilarczyk.
\newblock Nilpotent equilibria and chaos in tritrophic food chains.
\newblock {\em Preprint}, 2019.

\bibitem[DIR07]{DIR2007}
F.~Drubi, S.~Ib{\'a\~n}ez, and J.~A. Rodr{\'\i}guez.
\newblock Coupling leads to chaos.
\newblock {\em J. Differential Equations}, 239(2):371--385, 2007.

\bibitem[DIR19]{DIR2019}
F.~Drubi, S.~Ib{\'a\~n}ez, and D.~Rivela.
\newblock Chaotic behaviour in the unfolding of hopf-bogdanov-takens
  singularities.
\newblock {\em Preprint}, 2019.

\bibitem[Gas93]{Gas93}
P.~Gaspard.
\newblock Local birth of homoclinic chaos.
\newblock {\em Phys. D}, 62(1-4):94--122, 1993.
\newblock Homoclinic chaos (Brussels, 1991).

\bibitem[Gav87]{Gav}
N.~K. Gavrilov.
\newblock On some bifurcations of equilibria with a zero and a pair of purely
  imaginary roots (1978).
\newblock In {\em Methods of the qualitative theory of differential equations
  (Bifurcations of an equilibrium state with one zero root and a pair of purely
  imaginary roots and additional degeneration)}, pages 43--51. Gor'kov. Gos.
  Univ., Gorki, 1987.

\bibitem[GH90]{GH90}
J.~Guckenheimer and P.~Holmes.
\newblock {\em Nonlinear oscillations, dynamical systems, and bifurcations of
  vector fields}, volume~42.
\newblock Springer Verlag, 1990.

\bibitem[Guc81]{Guc81}
J.~Guckenheimer.
\newblock On a codimension two bifurcation.
\newblock In {\em Dynamical systems and turbulence, {W}arwick 1980 ({C}oventry,
  1979/1980)}, volume 898 of {\em Lecture Notes in Math.}, pages 99--142.
  Springer, Berlin-New York, 1981.

\bibitem[Hom02]{homburg2002}
A.~J. Homburg.
\newblock Periodic attractors, strange attractors and hyperbolic dynamics near
  homoclinic orbits to saddle-focus equilibria.
\newblock {\em Nonlinearity}, 15(4):1029--1050, 2002.

\bibitem[IR95]{ibarod1995}
S.~Ib{\'a}{\~n}ez and J.~A. Rodr{\'\i}guez.
\newblock Silnikov bifurcations in generic {$4$}-unfoldings of a
  codimension-{$4$} singularity.
\newblock {\em J. Differential Equations}, 120(2):411--428, 1995.

\bibitem[IR05]{ibarod2005}
S.~Ib{\'a}{\~n}ez and J.~A. Rodr{\'\i}guez.
\newblock Silnikov configurations in any generic unfolding of the nilpotent
  singularity of codimension three on {${\mathbb R}^3$}.
\newblock {\em J. Differential Equations}, 208(1):147--175, 2005.

\bibitem[JBL16]{JBL2016}
T.~J\'{e}z\'{e}quel, P.~Bernard, and E.~Lombardi.
\newblock Homoclinic orbits with many loops near a {$0^2i\omega$} resonant
  fixed point of {H}amiltonian systems.
\newblock {\em Discrete Contin. Dyn. Syst.}, 36(6):3153--3225, 2016.

\bibitem[Kuz98]{KUZ}
Y.~A. Kuznetsov.
\newblock {\em Elements of applied bifurcation theory}, volume 112 of {\em
  Applied Mathematical Sciences}.
\newblock Springer-Verlag, New York, second edition, 1998.

\bibitem[LTW05]{LTW}
J.~S.~W. Lamb, M.~A. Teixeira, and K.~N. Webster.
\newblock Heteroclinic bifurcations near {H}opf-zero bifurcation in reversible
  vector fields in {$\mathbb R^3$}.
\newblock {\em J. Differential Equations}, 219(1):78--115, 2005.

\bibitem[MV93]{MV93}
L.~Mora and M.~Viana.
\newblock Abundance of strange attractors.
\newblock {\em Acta Math.}, 171(1):1--71, 1993.

\bibitem[Poi90]{Poin1890}
H.~Poincar{\'e}.
\newblock {Sur le probleme ds trois corps et les \'equations de la dynamique}.
\newblock {\em Acta Math.}, 13:1--270, 1890.

\bibitem[PR97]{Puma97}
A.~Pumari{\~n}o and J.~A. Rodr{\'\i}guez.
\newblock {\em Coexistence and persistence of strange attractors}, volume 1658
  of {\em Lecture Notes in Mathematics}.
\newblock Springer-Verlag, Berlin, 1997.

\bibitem[PR01]{Puma01}
A.~Pumari{\~n}o and J.~A. Rodr{\'\i}guez.
\newblock Coexistence and persistence of infinitely many strange attractors.
\newblock {\em Ergodic Theory Dynam. Systems}, 21(5):1511--1523, 2001.

\bibitem[PT93]{PalisTakens}
J.~Palis and F.~Takens.
\newblock {\em Hyperbolicity and sensitive chaotic dynamics at homoclinic
  bifurcations}, volume~35 of {\em Cambridge Studies in Advanced Mathematics}.
\newblock Cambridge University Press, Cambridge, 1993.
\newblock Fractal dimensions and infinitely many attractors.

\bibitem[Shi65]{Shil65}
L.~P. Shilnikov.
\newblock {A case of the existence of a denumerable set of periodic motions}.
\newblock {\em Soviet Math. Dokl.}, 6:163--166, 1965.

\bibitem[Shi67]{Shil67}
L.~P. Shilnikov.
\newblock {The existence of a denumerable set of periodic motions in
  four-dimensional space in an extended neighborhood of a saddle-focus.}
\newblock {\em Soviet Math. Dokl.}, 8:54--58, 1967.

\bibitem[Sma67]{Sma67}
S.~Smale.
\newblock Differentiable dynamical systems.
\newblock {\em Bull. Amer. Math. Soc.}, 73:747--817, 1967.

\bibitem[Tak74]{Tak74}
F.~Takens.
\newblock Singularities of vector fields.
\newblock {\em Publications Math{\'e}matiques de l'IHES}, 43(1):47--100, 1974.

\bibitem[Tre84]{Tresser}
C.~Tresser.
\newblock About some theorems by {L}. {P}. {\v{s}}il'nikov.
\newblock {\em Ann. Inst. H. Poincar\'e Phys. Th\'eor.}, 40(4):441--461, 1984.

\end{thebibliography}


\end{document}